\documentclass[11pt,reqno]{amsart}

\usepackage{amsmath,amsthm,verbatim,amssymb,amsfonts,amscd, graphicx}
\usepackage{graphics}
\usepackage{geometry}
\theoremstyle{plain}

\newtheorem{defn}{Definition}[section]
\newtheorem{notation}[defn]{Notation}
\newtheorem{thm}[defn]{Theorem}
\newtheorem{cor}[defn]{Corollary}

\newtheorem{convention}[defn]{Convention}
\newtheorem{prop}[defn]{Proposition}
\newtheorem{remark}[defn]{Remark}
\newtheorem{lm}[defn]{Lemma}
\newtheorem{fact}[defn]{Fact}
\newtheorem{construction}[defn]{Construction}

\newtheorem{case}{Case}

\newtheorem{assume}[defn]{Assumption}
\newtheorem{ep}[defn]{Example}

\newtheorem*{main thm}{Theorem \ref{main}}

\newtheorem*{Dehn filling}{Proposition \ref{Dehn filling}}

\newtheorem*{pretzel}{Proposition \ref{pretzel}}
\newtheorem*{no interior singularity}{Proposition \ref{no interior singularity}}
\newtheorem*{application 1}{Proposition \ref{application 1}}
\newtheorem*{application 1a}{Proposition \ref{application 1} (a)}
\newtheorem*{application 1b}{Proposition \ref{application 1} (b)}
\newtheorem*{application 2}{Proposition \ref{application 2}}
\newtheorem*{part 1*}{Part I: $\tau$ realizes all rational slopes in $[0,\frac{p}{q+c})$}
\newtheorem*{part 2*}{Part II: $\tau$ realizes all rational slopes in $(-\infty, 0) \cup \{\infty\} \cup (\frac{p}{q-c}, +\infty)$}
\newtheorem*{part 1}{Part I: $\tau$ realizes the slope $0$}
\newtheorem*{part 2}{Part II: $\tau$ realizes all rational slopes in $(0,\frac{p}{q+c})$}
\newtheorem*{part 3}{Part III: $\tau$ realizes all rational slopes in $(-\infty, 0) \cup \{\infty\} \cup (\frac{p}{q-c}, +\infty)$}
\usepackage{extarrows}
\usepackage{mathrsfs}
\usepackage{galois}
\usepackage[colorlinks,citecolor = red, linkcolor=blue,hyperindex,CJKbookmarks]{hyperref}
\usepackage[all]{hypcap}
\usepackage[all]{xy}
\usepackage{tikz-cd}
\usepackage{latexsym}
\usepackage{subfigure}
\usepackage{amsaddr}

\newcommand\blfootnote[1]{%
	\begingroup
	\renewcommand\thefootnote{}\footnote{#1}%
	\addtocounter{footnote}{-1}%
	\endgroup
}

\begin{document}
	
\title[Taut foliations, Dehn fillings and co-orientation-reversing mapping tori]{co-orientable taut foliations in 
	Dehn fillings of pseudo-Anosov mapping tori with
	co-orientation-reversing monodromy}
\author{Bojun Zhao}
\address{Department of Mathematics, University at Buffalo, Buffalo, NY 14260, USA}
\maketitle

\begin{abstract}
	Let $\Sigma$ be a compact orientable surface with nonempty boundary,
	let $\varphi: \Sigma \to \Sigma$ be an orientation-preserving pseudo-Anosov homeomorphism,
	and let $M = \Sigma \times I / \stackrel{\varphi}{\sim}$ be the mapping torus of $\Sigma$ over $\varphi$.
	Let $\mathcal{F}^{s}$ denote the stable foliation of $\varphi$ in $\Sigma$.
	Let $T_1, \ldots, T_k$ denote the boundary components of $M$.
	With respect to a canonical choice of meridian and longitude on each $T_i$,
	the degeneracy locus of the suspension flow of $\varphi$ on $T_i$ can 
	be identified with a pair of integers $(p_i; q_i)$ such that
	$p_i > 0$ and $-\frac{1}{2}p_i < q_i \leqslant \frac{1}{2}p_i$.
	Let $c_i$ denote the number of components of $T_i \cap (\Sigma \times \{0\})$.
	Assume that $\mathcal{F}^{s}$ is co-orientable and
	$\varphi$ reverses the co-orientation on $\mathcal{F}^{s}$.
	We show that the Dehn filling of $M$ along $\partial M$ with
	any multislope in $J_1 \times \ldots \times J_k$ 
	admits a co-orientable taut foliation,
	where $J_i$ is one of the two open intervals in $\mathbb{R} \cup \{\infty\} \cong \mathbb{R}P^{1}$ between
	$\frac{p_i}{q_i + c_i}, \frac{p_i}{q_i - c_i}$ which
	doesn't contain $\frac{p_i}{q_i}$.

For some hyperbolic fibered knot manifolds,
the slopes given above contain all slopes that yield non-L-space Dehn filllings.
The examples include 
(1)
the exterior of the $(-2,3,2q+1)$-pretzel knot in $S^{3}$ for each $q \in \mathbb{Z}_{\geqslant 3}$
(previously proven by Krishna in 2020; see \hyperref[Kri]{[Kri]}),
(2)
the exteriors of many L-space knots in lens spaces.
\end{abstract}

\blfootnote{2020 \emph{Mathematics Subject Classification}.
	57M50.}
\blfootnote{Key words: Taut foliations, Dehn fillings, degeneracy slopes.}
\blfootnote{Email address: bojunzha@buffalo.edu}
\blfootnote{Present address: D\'epartement de math\'ematiques, 
	Universit\'e du Qu\'ebec \`a Montr\'eal, 201 President Kennedy Avenue,
	Montr\'eal, QC, Canada H2X 3Y7}

\section{Introduction}\label{section 1}

Throughout this paper, 
all $3$-manifolds are connected, orientable and irreducible.

The L-space conjecture (\hyperref[BGW]{[BGW]}, \hyperref[J]{[J]}) predicts that
the following statements are equivalent for a closed orientable irreducible $3$-manifold $M$:

(1)
$M$ is a non-L-space.

(2)
$\pi_1(M)$ is left-orderable.

(3)
$M$ admits a co-orientable taut foliation.

The implication (3)  $\Longrightarrow$ (1) is confirmed in \hyperref[OS1]{[OS1]}
(see also \hyperref[Bo]{[Bo]}, \hyperref[KazR2]{[KazR2]}).
In the case that $M$ has positive first Betti number,
(2) is proved in \hyperref[BRW]{[BRW]},
and (3) is proved in \hyperref[Ga1]{[Ga1]}.
The L-space conjecture has been verified when
$M$ is a graph manifold (\hyperref[BC]{[BC]}, \hyperref[Ra]{[Ra]}, \hyperref[HRRW]{[HRRW]}).

There are two common ways for constructing closed orientable $3$-manifolds:
Dehn surgeries on knots and links in $S^{3}$,
and Dehn fillings on mapping tori of compact orientable surfaces. 
Both of these two methods can produce all closed orientable $3$-manifolds
(\hyperref[A]{[A]}, \hyperref[Lic]{[Lic]}, \hyperref[W]{[W]}).
A slope on a knot manifold is called 
an \emph{NLS} or \emph{LO} or \emph{CTF filling slope} if the corresponding Dehn filling
satisfies (1) or (2) or (3), respectively,
and it is called an \emph{L-space filling slope} if the corresponding Dehn filling is an L-space.
Similarly, we use the terminologies \emph{NLS, LO, CTF, L-space filling multislopes} for
link manifolds (see Subsection \ref{subsection 2.1} for the convention of multislopes).
We also use the terminologies
\emph{NLS, LO, CTF, L-space surgery (multi)slopes} for knots or links.

For knot manifolds,
particularly knot exteriors in $S^{3}$,
much is known about the set of L-space filling slopes (\hyperref[KMOS]{[KMOS]}, \hyperref[RR]{[RR]}).
Naturally, 
one would expect to 

(1)
Find the LO and CTF filling slopes for knot manifolds and
compare them with the NLS filling slopes.

(2)
Explore the NLS, LO and CTF filling multislopes for link manifolds.

In this paper we focus on the problem of finding
CTF filling (multi)slopes on pseudo-Anosov mapping tori of compact orientable surfaces.
Let $\Sigma$ be a compact orientable surface with negative Euler characteristic and nonempty boundary, 
and let $\varphi: \Sigma \to \Sigma$ be an orientation-preserving pseudo-Anosov homeomorphism.
Let $M = \Sigma \times I / \stackrel{\varphi}{\sim}$ be the mapping torus of $\Sigma$ over $\varphi$,
i.e. the quotient space of $\Sigma \times [0,1]$ under the equivalence relation
$(x,1) \sim (\varphi(x),0)$ for all $x \in \Sigma$.
We fix an orientation on $M$,
and we endow $\Sigma$ with an orientation induced from $M$ (see Convention \ref{orientation} (a) for details).

We choose a canonical meridian/longitude coordinate system on each component $T$ of $\partial M$,
following \hyperref[Ro2]{[Ro2]}.
See Convention \ref{slope} for details on the choice of the coordinate system,
and see Convention \ref{orientation} (b), (c) for their orientations.
If $M$ is the exterior of a fibered knot $K$ in $S^{3}$,
this coordinate system coincides with
the standard meridian/longitude coordinate system except for a special case
(see Remark \ref{slope consistent} for details).
With respect to the canonical coordinate system,
a slope on each component $T$ of $\partial M$ will be identified with
an element of $\mathbb{Q} \cup \{\infty\}$ as usual.

Let $T$ be a component of $\partial M$.
The suspension flow of $\varphi$ in $M$,
denoted $\Psi$,
has $2n$ closed orbits on $T$ for some $n \in \mathbb{N}_+$,
which are parallel essential simple closed curves on $T$.
The slope of these closed orbits is referred to as 
the \emph{degeneracy slope} of $T$,
which we denote by $\delta_T$.
The \emph{degeneracy locus} of $\Psi$ on $T$,
denoted $d(T)$,
is identified with $(nu; nv)$ for $u, v \in \mathbb{Z}$ such that
$\frac{u}{v} = \delta_{T}$,
$u > 0$, and $\gcd(u, v) = 1$
(where $u = 1, v = 0$ if $\delta_T = \infty$).
The integer $n$ is referred to as the \emph{multiplicity} of $d(T)$.
In the case where $\varphi$ preserves all boundary components of $\Sigma$,
$\varphi$ is said to be \emph{right-veering} (resp. \emph{left-veering}) if
$\delta_T > 0$ (resp. $\delta_T < 0$) for each component $T$ of $\partial M$.
See Subsection \ref{degeneracy locus} for
a more detailed explanation and some relevant references.

Let $T$ be a component of $\partial M$.  
Under the canonical coordinate system on $\partial M$ (see Convention \ref{slope} for details),  
the degeneracy slope $\delta_T$ lies in  
$(\mathbb{Q} \cup \{\infty\}) - [-2, 2)$.
In addition,
if $d(T) = (p;q)$,
then $-\frac{1}{2}p < q \leqslant \frac{1}{2}p$,
$\frac{p}{q} = \delta_{T}$,
and $\gcd(p,q)$ is the multiplicity of $d(T)$.

We adopt the following conventions for (multi)slopes and Dehn fillings:

\begin{convention}\rm\label{filling convention}
	(a)
	Let $k$ denote the number of boundary components of $M$.
	For any multislope $\textbf{s} \in (\mathbb{Q} \cup \{\infty\})^{k}$ of $\partial M$,
	$M(\textbf{s})$ denotes the Dehn filling of $M$ along $\partial M$ with the multislope $\textbf{s}$.
	
	(b)
	Throughout this paper,
	for a slope $\frac{p}{q}$ on a component $T$ of $\partial M$,
	we allow $\gcd(p,q) > 1$.
	In this case,
	we consider $\frac{p}{q}$ as
	the corresponding fraction in reduced form,
	i.e.
	$\frac{p}{q}$ is identified with $\frac{u}{v}$ for which
	$u = \frac{p}{\gcd(p,q)}$,
	$v = \frac{q}{\gcd(p,q)}$.
\end{convention}

Now we review some known results on 
CTF filling slopes of $M$.
To be convenient,
we state the following two theorems in our setting,
although they do not need to restrict $\varphi$ to be pseudo-Anosov.

The following theorem is proved by Roberts in \hyperref[Ro1]{[Ro1]}, \hyperref[Ro2]{[Ro2]},
originally stated in \hyperref[Ro2]{[Ro2, Theorem 4.7]} for the case where $\varphi$ is pseudo-Anosov.
For a version that does not require $\varphi$ to be pseudo-Anosov, see \hyperref[DR]{[DR, Theorem 1.4]}.

\begin{thm}[Roberts]\label{Roberts}
	Suppose that $\Sigma$ has exactly one boundary component.
	
	(a)
	If $\varphi$ is right-veering,
	then $M(s)$ admits a co-orientable taut foliation for any rational slope $s \in (-\infty, 1)$.
	
	(b)
	If $\varphi$ is left-veering,
	then $M(s)$ admits a co-orientable taut foliation for any rational slope $s \in (-1, +\infty)$.
	
	(c)
	If $\varphi$ is neither right-veering nor left-veering,
	then $M(s)$ admits a co-orientable taut foliation for any rational slope $s \in (-\infty, +\infty)$.
	
	Moreover,
	the core curve of the filling solid torus is transverse to the foliation in each case.
\end{thm}

For the case that 
$\Sigma$ has multiple boundary components,
Kalelkar and Roberts (\hyperref[KalR]{[KalR]}) prove 

\begin{thm}[Kalelkar-Roberts]\label{multi}
	Let $k$ denote the number of boundary components of $M$.
	There is a neighborhood $J \subseteq \mathbb{R}^{k}$ of $(0,\ldots,0)$ such that
	for any rational multislope $(s_1,\ldots,s_k) \in J$,
	$M(s_1,\ldots,s_k)$ admits a co-orientable taut foliation.
	Moreover,
	the core curves of the filling solid tori are transverse to the foliation.
\end{thm}

Let $\mathcal{F}^{s}, \mathcal{F}^{u}$ denote the stable and unstable foliations of $\varphi$.
The monodromy $\varphi$ is said to be \emph{co-orientable} if $\mathcal{F}^{s}$ is co-orientable,
and $\varphi$ is said to be \emph{co-orientation-preserving} (resp. \emph{co-orientation-reversing}) if
$\varphi$ is co-orientable and
preserves (resp. reverses) the co-orientation on $\mathcal{F}^{s}$.
Here, 
$\mathcal{F}^{s}$ can be replaced with $\mathcal{F}^{u}$
(see Remark \ref{independent}).
For more information on co-orientable pseudo-Anosov homeomorphisms,
see Subsection \ref{co-orientable}.

For general pseudo-Anosov mapping tori of compact surfaces with more than one boundary component,
co-orientation-preserving is the only known case to have
an explicit multi-interval such that all rational multislopes in it are CTF filling multislopes.
The following theorem is implicitly contained in \hyperref[Ga3]{[Ga3]}.
For the reader’s convenience, we briefly sketch how it is deduced from Gabai’s argument.
By pushing $\partial M$ slightly inward, we obtain from the suspension foliation of $\mathcal{F}^{s}$
a foliation that meets each boundary torus in Reeb annuli separated by circles
(see \hyperref[Ga3]{[Ga3, Operation 2.3.1]}).
Likewise, for each interior singular orbit of the suspension flow of $\varphi$,
we can choose a solid torus whose core is precisely this singular orbit
and whose boundary torus intersects the suspension foliation of $\mathcal{F}^{s}$ in the same way,
namely in Reeb annuli separated by circles.
Removing all such solid tori from $M$,
the resulting manifold satisfies the hypotheses of \hyperref[Ga3]{[Ga3, Operation 2.4.1]},
which ensures that any Dehn filling whose restricted slope on each boundary component
is not isotopic to the core of a Reeb annulus admits a co-orientable taut foliation.
Finally, by assigning to every boundary torus coming from an interior singular orbit
the filling slope corresponding to that of the original solid torus,
we obtain the desired result for Dehn fillings of $M$.

\begin{thm}[Gabai]\label{co-orientation-preserving}
	Assume that $\varphi$ is co-orientable and co-orientation-preserving.
	Let $k$ denote the number of boundary components of $M$.
	For any multislope $(s_1, \ldots, s_k) \in (\mathbb{Q} \cup \{\infty\})^{k}$ such that
	each $s_i$ is not the degeneracy slope of the corresponding boundary component,
	$M(s_1, \ldots, s_k)$ admits a co-orientable taut foliation.
	Moreover,
	the core curves of the filling solid tori are transverse to the foliation.
\end{thm}

As an immediate consequence of the above result, 
for any hyperbolic fibered knot in $S^{3}$ with co-orientation-preserving monodromy, 
the degeneracy slope is the meridian; this conclusion also follows from \hyperref[GO]{[GO, Theorem~5.3]}.

Unlike the co-orientation-preserving case,
many $M(s_1,\ldots,s_k)$ (with each $s_i$ being other than the degeneracy slope) 
admit no co-orientable taut foliation when $\varphi$ is co-orientation-reversing,
since there are many hyperbolic fibered L-space knots in $S^{3}$ and 
lens spaces that 
have co-orientation-reversing monodromy
(see Proposition \ref{pretzel}, Examples \ref{v0751}$\sim$\ref{example1}).
We concentrate on the co-orientation-reversing case in this paper.

\subsection{The main results}\label{subsection 1.1}

Let $T_1, \ldots, T_k$ denote the boundary components of $M$.
For each $i \in \{1, \ldots, k\}$,
we choose a boundary component $C_i$ of $\Sigma$ such that $C_i \times \{0\} \subseteq T_i$,
and let $c_i$ denote the order of $C_i$ under $\varphi$
(i.e. $c_i = \min \{k \in \mathbb{N}_+ \mid \varphi^{k}(C_i) = C_i\}$).
Let $(p_i; q_i)$ denote the degeneracy locus of the suspension flow of $\varphi$ on each $T_i$.
We note that if $\varphi$ is co-orientable,
then $2 \mid p_i$,
and moreover,
if $\varphi$ is co-orientation-reversing,
then $q_i \equiv c_i \text{ } (\text{mod } 2)$
(this is because $2 \mid q_i$ if and only if $\varphi^{c_i}$ is co-orientation-preserving,
see Remark \ref{remark} (a)).

\begin{thm}\label{main}
	Assume that $\varphi$ is co-orientable and co-orientation-reversing
	(then $2 \mid p_i$ and $q_i \equiv c_i \text{ } (\emph{mod } 2)$ for each $i$).
	Then $M(s_1,\ldots,s_k)$ admits a co-orientable taut foliation for 
	any multislope $(s_1,\ldots,s_k) \in (J_1 \times \ldots \times J_k) \cap (\mathbb{Q} \cup \{\infty\})^{k}$,
	where each $J_i$ is a set of slopes on $T_i$ such that
	
	\begin{equation*}
	J_i =
	\begin{cases}
	(-\infty, \frac{p_i}{q_i + c_i}) \cup 
	(\frac{p_i}{q_i - c_i}, +\infty) \cup \{\infty\} & \emph{if }  q_i > c_i > 0 \\
	(-\infty, \frac{p_i}{2q_i}) & \emph{if }  q_i = c_i > 0 \\
	(-\frac{p_i}{c_i - q_i}, \frac{p_i}{q_i + c_i}) & \emph{if } c_i > q_i \geqslant 0 \\
	(-\frac{p_i}{|q_i| + c_i}, \frac{p_i}{c_i - |q_i|}) & \emph{if } - c_i < q_i < 0 \\
	(-\frac{p_i}{2|q_i|}, +\infty) & \emph{if } q_i = -c_i < 0  \\
	(-\infty, -\frac{p_i}{|q_i| - c_i}) \cup 
	(-\frac{p_i}{|q_i| + c_i}, +\infty) \cup \{\infty\} & \emph{if } q_i < -c_i < 0.
	\end{cases}
\end{equation*}
     Moreover,
     the core curves of the filling solid tori are transverse to the leaves of these foliations.
\end{thm}

In the case that $\varphi$ preserves each component of $\partial \Sigma$,
$c_i = 1$ for each $1 \leqslant i \leqslant k$.

\begin{cor}\label{open book}
	Assume that $\varphi$ is co-orientable, co-orientation-reversing and
	$\varphi(C) = C$ for each component $C$ of $\partial \Sigma$.
	Then $M(s_1,\ldots,s_k)$ admits a co-orientable taut foliation for 
	any multislope $(s_1,\ldots,s_k) \in (J_1 \times \ldots \times J_k) \cap (\mathbb{Q} \cup \{\infty\})^{k}$,
	where each $J_i$ is a set of slopes on $T_i$ such that
	
	\begin{equation*}
		J_i =
		\begin{cases}
			(-\infty, \frac{p_i}{q_i + 1}) \cup (\frac{p_i}{q_i - 1}, +\infty) \cup \{\infty\} & \emph{if } q_i > 1 \\
			(-\infty, \frac{p_i}{2}) & \emph{if }  q_i = 1 \\
			(-\frac{p_i}{2}, +\infty) & \emph{if } q_i = -1 \\
			(-\infty, -\frac{p_i}{|q_i| - 1}) \cup (-\frac{p_i}{|q_i| + 1}, +\infty) \cup \{\infty\} & \emph{if } q_i < -1.
		\end{cases}
		\end{equation*}
\end{cor}

\begin{remark}\rm\label{remark}
	(a)
	As explained later in Convention \ref{orientation} (b),
	$C_i$ is assigned an orientation induced from $\Sigma$;
	and in Convention \ref{orientation} (c),
	the longitude of $T_i$ is oriented consistently with $C_i$, 
	and the meridian of $T_i$ is oriented consistently with
	the increasing orientation of the second coordinate in $\Sigma \times I / \sim$.
	Note that $p_i$ is the number of singular points of $\mathcal{F}^{s}$ contained in $C_i$.  
	Let $v_1, \ldots, v_{p_i}$ be these singular points, ordered consecutively along the orientation of $C_i$.  
	Then $\varphi^{c_i}(v_j) = v_{j+q_i}$ (mod $p_i$) for each $1 \leqslant j \leqslant p_i$.
	
	(b)
	Assume that $\varphi$ is co-orientable, co-orientation-reversing
	and $\partial \Sigma$ is connected.
	Then the degeneracy slope of $\partial M$ is not $\infty$,
	and thus,
	one of Theorem \ref{Roberts} (a), (b) applies to $M$ and
	Theorem \ref{Roberts} (c) doesn't apply.
	In this case,
	Theorem \ref{main} expands the known range of CTF filling slopes when
	$(p_1;q_1) \ne (2;1)$.
	
	(c)
	In Theorem \ref{main},
	if we regard the set of slopes $\mathbb{R} \cup \{\infty\}$ on each $T_i$ as
	$\mathbb{R}P^{1} \cong S^{1}$,
	then $J_i$ is one of the two open intervals in $\mathbb{R} \cup \{\infty\}$ between
	$\frac{p_i}{q_i + c_i}, \frac{p_i}{q_i - c_i}$ that doesn't contain 
	the degeneracy slope $\frac{p_i}{q_i}$ on $T_i$.
	For convenience,
	we will call $J_i$ an interval and
	call $J_1 \times \ldots \times J_k$ a multi-interval.
\end{remark}

\subsection{Applications to Dehn surgeries on knots}\label{subsection 1.2}

In this subsection, 
we apply Corollary \ref{open book} to Dehn surgeries on knots in $S^{3}$ or spherical manifolds.
At first,
we describe some properties of L-space filling slopes on knot manifolds.

A knot manifold is \emph{Floer simple} if
it has at least two distinct L-space filling slopes
(compare with \hyperref[RR]{[RR, Proposition 1.3]}),
and a knot in a closed $3$-manifold is called an \emph{L-space knot} if its exterior is Floer simple.
We note from \hyperref[RR]{[RR, Theorem 1.6]} that for any Floer simple knot manifold $N$,
there is an interval $\mathcal{L}(N)$ of slopes such that 
$$\{\text{L-space filling slopes of } N\} = \mathcal{L}(N) \cap (\mathbb{Q} \cup \{\infty\}),$$
and in particular,
$\mathcal{L}(N)$ either consists of all slopes except the homological longitude,
or is a closed interval.
We call $\mathcal{L}(N)$ the \emph{maximal L-space filling interval} of $N$.
If $\mathcal{L}(N)$ is a closed interval,
then we call its complement the \emph{maximal NLS filling interval} of $N$.

A knot in a closed $3$-manifold is called a \emph{fibered knot} if
its exterior is a once-punctured surface bundle over $S^{1}$.
Note that all L-space knots in $S^{3}$ are fibered
(\hyperref[Gh]{[Gh]}, \hyperref[Ni]{[Ni]}).
Now, 
we assume that $M$ is the exterior of
a hyperbolic fibered L-space knot $K$ in a closed orientable $3$-manifold $W$,
and that $W$ is either $S^{3}$ or a spherical manifold.
It's known from Theorem \ref{Roberts} and \hyperref[OS1]{[OS1]} that
$\varphi$ is either left-veering or right-veering.
Let $T = \partial M$,
and let $d(T) = (p;q)$ (then $\delta_T = \frac{p}{q}$).
Let $m_K$ denote the slope on $T$ with 
$M(m_K) = W$,
and let $\lambda_T$ denote the longitude of $T$.
It's known from \hyperref[GO]{[GO, Theorem 5.3]} that 
$$\Delta(m_K, d(T)) = \gcd(p, q) \Delta(m_K, \delta_T) < 2$$
since $W$ admits no essential lamination,
where $\Delta(m_K, d(T)), \Delta(m_K, \delta_T)$ denote
the minimal geometric intersection numbers of $m_K, d(T)$ and
of $m_K, \delta_T$.
Thus $m_K$ belongs to one of the following three cases:

\begin{case}\rm\label{case1}
	$m_K = \delta_T$.
\end{case}

\begin{case}\rm\label{case2}
	$\Delta(m_K, \lambda_T) = \Delta(m_K, \delta_T) = 1$.
	It's not hard to observe from $\delta_T \notin [-2,2)$ that
	$m_K \in \{\infty, 1\}$.
	Then either 
	(1) $m_K = \infty$ or
	(2) $m_K = 1$ and $\delta_T = 2$.
	In the possibility (2),
	we can choose the opposite orientation on $M$ to make $m_K$ become $\infty$,
	under the coordinate system induced from this orientation 
	(see Conventions \ref{slope}, \ref{orientation}).
	In the possibility (1),
	$|q| = 1$ and
	$2 \leqslant p \leqslant 4g-2$ (\hyperref[Ga4]{[Ga4]}).
	Thus
	$d(T)$ has multiplicity $1$.
\end{case}

\begin{case}\rm\label{case3}
	$\Delta(m_K, \lambda_T) > \Delta(m_K, \delta_T) = 1$.
	In this case,
	$m_K \notin \{\delta_T, \infty, 1\}$,
	and
	$d(T)$ has multiplicity $1$.
\end{case}

We call $K$ a \emph{type-I} (resp. \emph{type-II}, \emph{type-III}) knot in $W$ if
Case \ref{case1} (resp. Case \ref{case2}, Case \ref{case3}) holds.
As explained in Case \ref{case2},
we may always assume $m_K = \infty$ when $K$ is a type-II knot in $W$.
Let $g$ denote the fibered genus of $M$.
When $W = S^{3}$,
$K$ can only be a type-II knot,
and the maximal NLS filling interval of $M$ is
$(-2g+1, +\infty)$ if
$\varphi$ is left-veering and
is $(-\infty, 2g-1)$ if
$\varphi$ is right-veering.
This result is essentially known from \hyperref[KMOS]{[KMOS]} and \hyperref[OS3]{[OS3]}
(see also \hyperref[BaS]{[BaS, Theorem 1.15]}).

However,
when $W$ is a lens space and $K$ is a type-II knot in $W$,
the maximal NLS filling interval of $M$ doesn't have to be
$(-2g+1, +\infty)$ or $(-\infty, 2g-1)$
(for instance,
Examples \ref{v0751}, \ref{example1}).
Also,
$K$ can be a type-I knot or a type-III knot when $W$ is a lens space
(Examples \ref{v0751}, \ref{general}).

\begin{remark}\rm\label{veering}
	We note that if $\varphi$ is left-veering (resp. right-veering),
	then we can choose an opposite orientation on $W$ to 
	get a right-veering (resp. left-veering) monodromy of $M$.
	In the case of $\delta_T \ne 2$,
	the meridian on $T$ doesn't change,
	and any other slope $s \in \mathbb{Q}$ becomes $-s$ under
	the canonical coordinate system induced from the opposite orientation on $M$.
\end{remark}

We give some examples in the remainder of this subsection.
For $q \in \mathbb{N}$ with $q \geqslant 3$,
the $(-2,3,2q+1)$-pretzel knot $K$ in $S^{3}$ 
is a hyperbolic L-space knot with Seifert genus $g(K) = q+2$.
See \hyperref[Ga2]{[Ga2, Theorem 6.7]} for the proof that $K$ is fibered 
(and for a geometric description of its fibered surface, from which the genus can be verified),
\hyperref[O1]{[O1, Corollary 5]} for the hyperbolicity of $K$,
and \hyperref[OS2]{[OS2, page~1291]} (see also \hyperref[LM]{[LM, Theorem 1]}) for the fact that $K$ is an L-space knot.
We have

\begin{prop}\label{pretzel}
	Let $K$ denote the $(-2,3,2q+1)$-pretzel knot in $S^{3}$ with $q \geqslant 3$.
	We fix an orientation on $S^{3}$ so that $K$ has right-veering monodromy.
	
	(a)
	$K$ has co-orientable and co-orientation-reversing monodromy.
	
	(b)
	The degeneracy slope on $K$ is $4g(K)-2$.
	
	(c)
	All rational slopes in $(-\infty, 2g(K) - 1)$ are CTF surgery slopes of $K$.
	
	(d)
	For each $n \in \mathbb{N}$ with $n \geqslant 2$,
	the $n$-fold cyclic branched cover of $S^{3}$ over $K$,
	denoted $\Sigma_n(K)$,
	admits a co-orientable taut foliation.
\end{prop}

\begin{remark}\rm
	(1)
	Proposition \ref{pretzel} (c) has been proved in \hyperref[Kri]{[Kri]}.
	
	(2)
	In Proposition \ref{pretzel} (d),
	the foliations are obtained from Theorem \ref{co-orientation-preserving}
	when $n$ is even and from Corollary \ref{open book} when $n$ is odd.
	When $n \geqslant 4g(K) - 2 = 4(q+2)-2 = 4q+6$,
	it can be deduced from Theorem \ref{Roberts} that 
	$\Sigma_n(K)$ admits a co-orientable taut foliation
	(see \hyperref[BH]{[BH]} for an explanation).
	It follows from (d) and \hyperref[OS1]{[OS1]} that
	$\Sigma_n(K)$ is a non-L-space.
	In fact,
	it's already known from \hyperref[BBG]{[BBG, Corollary 1.2]} that
	$\Sigma_n(K)$ is a non-L-space
	(see \hyperref[Nie]{[Nie, Subsection 3.2]} for an explanation that
	\hyperref[BBG]{[BBG, Corollary 1.2]} can be applied to $\Sigma_n(K)$).
	By \hyperref[BGH]{[BGH, Theorem 1.2]},
	$\Sigma_n(K)$ also has left-orderable fundamental group.
\end{remark}

It follows that

\begin{cor}\label{Pretzel}
	For every $(-2,3,2q+1)$-pretzel knot $K$ in $S^{3}$ with $q \geqslant 3$,
    and for each $n \geqslant 2$,
    the $n$-fold cyclic branched cover of $K$ satisfies 
	each of (1), (2), (3) of the L-space conjecture.
\end{cor}

Furthermore,

\begin{prop}\rm\label{no interior singularity}
	Let $K$ be a hyperbolic L-space knot in $S^{3}$ with co-orientation-reversing monodromy.
	If the stable foliation of its monodromy has no singularity in the interior of the fibered surface,
	then a manifold obtained from Dehn surgery on $K$ admits a co-orientable taut foliation
	if and only if it is a non-L-space.
\end{prop}

In the census \hyperref[D3]{[D3]}, 
there are $3242$ $1$-cusped hyperbolic fibered $3$-manifolds,
each of which is the mapping torus of a once-punctured surface such that
the stable foliation of its monodromy only has singularities with an even number of prongs.
Among them,
$805$ have co-orientation-preserving monodromy,
$2214$ have co-orientation-reversing monodromy,
and the remaining $223$ don't have co-orientable monodromy.
Let $\mathcal{N}$ denote the set of these $2214$ manifolds with 
co-orientation-reversing monodromy.
For each $N \in \mathcal{N}$,
we can obtain a CTF filling interval from Corollary \ref{open book}.
In \hyperref[D1]{[D1]},
the manifolds in $\mathcal{N}$ are tested for being Floer simple,
and some Dehn fillings of them are tested for being L-spaces.
We can find some L-space filling slopes and NLS filling slopes of the manifolds in $\mathcal{N}$ from
\hyperref[D2]{[D2]} (the data associated to \hyperref[D1]{[D1]}).

We list some examples in $\mathcal{N}$ below with
explicit CTF filling intervals obtained from Corollary \ref{open book},
as well as other data,
such as fibered genera, degeneracy slopes, and maximal NLS filling intervals.
We first explain how this data is obtained.
For each $1$-cusped manifold $N$ in Examples \ref{v0751}$\sim$\ref{nonsharp},

$\bullet$
The fibered genus and the degeneracy slope of $N$ are known from \hyperref[D3]{[D3]}.
By comparing \hyperref[D2]{[D2]} with
snapPy \hyperref[CDGW]{[CDGW]},
we can know that $N$ is the complement of a hyperbolic fibered L-space knot in 
$S^{3}$ or some spherical manifold of
type-II or type-III
($N$ is also possible to be the complement of a type-I knot in another spherical manifold).
So the degeneracy locus of the suspension flow has multiplicity $1$ (see Cases \ref{case2}, \ref{case3}).
From the degeneracy slope of $N$,
we can obtain a CTF filling interval from Corollary \ref{open book} directly.

$\bullet$
If $N$ is contained in Examples \ref{v0751}$\sim$\ref{nonspherical},
then the two endpoints of the obtained CTF filling interval of $N$ are verified to be
L-space filling slopes in \hyperref[D2]{[D2]}.
Combined with \hyperref[OS1]{[OS1]} and \hyperref[RR]{[RR]},
the obtained CTF filling interval of $N$ is exactly the maximal NLS filling interval of $N$.

Let $L(p,q)$ denote the lens space obtained from
Dehn surgery along the unknot in $S^{3}$ with the slope $-\frac{p}{q}$.
If we write $L(p,q)$ such that $p$ is specified but $q$ is not specified,
that means it's valid for the given $p$ and some $q$,
for example,
$L(5,q)$ denotes an unspecified element of $\{L(5,1), L(5,2)\}$.
Also note that
all slopes in the following examples are consistent with Convention \ref{slope} 
(instead of the slopes in snapPy).
Similar to Convention \ref{filling convention},
for any $1$-cusped $3$-manifold $N$,
we denote by $N(s)$ the Dehn filling of $N$ with slope $s$ for any 
$s \in \mathbb{Q} \cup \{\infty\}$.

\begin{ep}\rm\label{v0751}
	In Table \ref{table 1},
	we list some examples in $\mathcal{N}$, 
	each of which has three distinct lens space Dehn fillings,
	and it can be identified with 
	the complement of a type-I$\sim$III knot in these three lens spaces respectively.
	The obtained CTF filling intervals of these examples are equal to
	their maximal NLS filling intervals.
	
	\begin{table}[h]\label{table 1}
		\caption{Some examples with three distinct lens space Dehn fillings.}
	\begin{tabular}{| l | l | l | l | l |}
		\hline
		manifold & genus & degeneracy slope & obtained CTF filling interval & maximal NLS filling interval \\
		\hline
		$m122$ & $g = 2$ & $\delta = 4$ & $(-\infty, 2)$ 
		& $(-\infty, 2)$  \\ \cline{2-5}
		& \multicolumn{4}{| l |}{$m122(4) = L(28, q_1)$, 
			$m122(5) = L(35, q_2)$, 
			$m122(\infty) = L(7, q_3)$.}  \\ 
		\hline
		$m280$ & $g = 2$ & $\delta = -4$ & $(-2, +\infty)$ 
		& $(-2, +\infty)$  \\ \cline{2-5}
		& \multicolumn{4}{| l |}{$m280(-4) = L(44, q_1)$, 
			$m280(-3) = L(33, q_2)$, 
			$m280(\infty) = L(11, q_3)$.}  \\ 
		\hline
		$v0751$ & $g = 3$ & $\delta = -6$ & $(-3, +\infty)$ 
		& $(-3, +\infty)$  \\ \cline{2-5}
		& \multicolumn{4}{| l |}{$v0751(-6) = L(78, q_1)$, 
			$v0751(-5) = L(65, q_2)$, 
			$v0751(\infty) = L(13, q_3)$.}  \\ 
		\hline
	\end{tabular}
\end{table}
\end{ep}

\begin{ep}\rm\label{general}
	In Table \ref{table 2},
	we give some more examples in $\mathcal{N}$ which are complements of type-III knots in lens spaces, 
	and the obtained CTF filling intervals of them are equal to
	their maximal NLS filling intervals.
	
	\begin{table}[h]\label{table 2}
		\caption{Some examples which are complements of type-III knots in lens spaces.}
	\begin{tabular}{| l | l | l | l | l |}
		\hline
		manifold & genus & degeneracy slope & obtained CTF filling interval & maximal NLS filling interval \\
		\hline
		$s297$ & $g = 2$ & $\delta = -6$ & $(-3,+\infty)$ & $(-3,+\infty)$ \\ \cline{2-5}
		& \multicolumn{4}{| l |}{$s297(-5) = L(30,q)$.}  \\ 
		\hline
		$s408$ & $g = 3$ & $\delta = -8$ & $(-4,+\infty)$ & $(-4,+\infty)$ \\ \cline{2-5}
		& \multicolumn{4}{| l |}{$s408(-7) = L(42,q)$.}  \\ 
		\hline
		$o9_{26541}$ & $g = 3$ & $\delta = -\frac{8}{3}$ & $(-\infty, -4) \cup (-2, +\infty) \cup \{\infty\}$
		& $(-\infty, -4) \cup (-2, +\infty) \cup \{\infty\}$ \\ \cline{2-5}
		& \multicolumn{4}{| l |}{$o9_{26541}(-3) = L(87,q)$.}  \\ 
		\hline
	\end{tabular}
\end{table}
\end{ep}

\begin{ep}\rm\label{example1}
	In Table \ref{table 3},
	we list some examples in $\mathcal{N}$,
	each of which is the complement of a type-II knot in some $L(p,q)$ with relatively small $p$,
	and its obtained CTF filling interval is exactly
	its maximal NLS filling interval.
	
	\begin{table}[h]\label{table 3}
		\caption{Some examples which are complements of type-II knots in $L(p,q)$ with relatively small $p$.}
	\begin{tabular}[H]{| l | l | l | l | l | l |}
		\hline
		manifold & genus & degeneracy & obtained CTF filling & maximal NLS filling & lens space filling  \\
		& & slope & interval & interval &   \\
		\hline
		$m146$ & $g = 3$ & $\delta = -10$ & $(-5, +\infty)$ & $(-5, +\infty)$ & $m146(\infty) = \mathbb{R}P^{3}$ \\ \cline{2-6}
		\hline
		$v2585$ & $g = 4$ & $\delta = 14$ & $(-\infty,7)$ & $(-\infty, 7)$ & $v2585(\infty) = \mathbb{R}P^{3}$ \\ \cline{2-6}
		\hline
		$m036$ & $g = 2$ & $\delta = -6$ & $(-3, +\infty)$ & $(-3, +\infty)$ & $m036(\infty) = L(3,1)$ \\ \cline{2-6}
		\hline
		$s313$ & $g = 3$ & $\delta = -10$ & $(-5, +\infty)$ & $(-5, +\infty)$ & $s313(\infty) = L(3,1)$ \\ \cline{2-6}
		\hline
		$v3327$ & $g = 3$ & $\delta = 10$ & $(-\infty, 5)$ & $(-\infty, 5)$ & $v3327(\infty) = L(3,1)$ \\ \cline{2-6}
		\hline
		$v3248$ & $g = 4$ & $\delta = -14$ & $(-7, +\infty)$ & $(-7, +\infty)$ & $v3248(\infty) = L(3,1)$ \\ \cline{2-6}
		\hline
		$o9_{36980}$ & $g = 4$ & $\delta = 14$ & $(-\infty, 7)$ & $(-\infty, 7)$ & $o9_{36980}(\infty) = L(3,1)$ \\ \cline{2-6}
		\hline
		$s479$ & $g = 2$ & $\delta = 6$ & $(-\infty, 3)$ & $(-\infty, 3)$ & $s479(\infty) = L(4,1)$ \\ \cline{2-6}
		\hline
		$v2296$ & $g = 3$ & $\delta = -10$ & $(-5, +\infty)$ & $(-5, +\infty)$ & $v2296(\infty) = L(4,1)$ \\ \cline{2-6}
		\hline
		$t11938$ & $g = 3$ & $\delta = -10$ & $(-5, +\infty)$ & $(-5, +\infty)$ & $t11938(\infty) = L(4,1)$ \\ \cline{2-6}
		\hline
		$t11829$ & $g = 4$ & $\delta = -14$ & $(-7, +\infty)$ & $(-7, +\infty)$ & $t11829(\infty) = L(4,1)$ \\ \cline{2-6}
		\hline
		$o9_{39567}$ & $g = 5$ & $\delta = -18$ & $(-9, +\infty)$ & $(-9, +\infty)$ & $o9_{39567}(\infty) = L(4,1)$ \\ \cline{2-6}
		\hline
		$v1682$ & $g = 2$ & $\delta = -6$ & $(-3, +\infty)$ & $(-3, +\infty)$ & $v1682(\infty) = L(5,q)$ \\ \cline{2-6}
		\hline
		$t07148$ & $g = 3$ & $\delta = -10$ & $(-5, +\infty)$ & $(-5, +\infty)$ & $t07148(\infty) = L(5,q)$ \\ \cline{2-6}
		\hline
		$t09415$ & $g = 3$ & $\delta = 8$ & $(-\infty, 4)$ & $(-\infty, 4)$ & $t09417(\infty) = L(5,q)$ \\ \cline{2-6}
		\hline
		$o9_{35581}$ & $g = 3$ & $\delta = -10$ & $(-5, +\infty)$ & $(-5, +\infty)$ & $o9_{35581}(\infty) = L(5,q)$ \\ \cline{2-6}
		\hline
		$t06948$ & $g = 4$ & $\delta = 12$ & $(-\infty, 6)$ & $(-\infty, 6)$ & $t06948(\infty) = L(5,q)$ \\ \cline{2-6}
		\hline
		$o9_{35298}$ & $g = 4$ & $\delta = 14$ & $(-\infty, 7)$ & $(-\infty, 7)$ & $o9_{35298}(\infty) = L(5,q)$ \\ \cline{2-6}
		\hline
		$t04406$ & $g = 2$ & $\delta = -6$ & $(-3, +\infty)$ & $(-3, +\infty)$ & $t04406(\infty) = L(6,1)$ \\ \cline{2-6}
		\hline
		$o9_{18836}$ & $g = 3$ & $\delta = -10$ & $(-5, +\infty)$ & $(-5, +\infty)$ & $o9_{18836}(\infty) = L(6,1)$ \\ \cline{2-6}
		\hline
		$v2678$ & $g = 2$ & $\delta = 6$ & $(-\infty, 3)$ & $(-\infty, 3)$ & $v2678(\infty) = L(7,q)$ \\ \cline{2-6}
		\hline
		$o9_{10415}$ & $g = 2$ & $\delta = -6$ & $(-3, +\infty)$ & $(-3, +\infty)$ & $o9_{10415}(\infty) = L(7,q)$ \\ \cline{2-6}
		\hline
		$t05634$ & $g = 3$ & $\delta = 8$ & $(-\infty, 4)$ & $(-\infty, 4)$ & $t05634(\infty) = L(7,q)$ \\ \cline{2-6}
		\hline
		$o9_{14062}$ & $g = 3$ & $\delta = -6$ & $(-3, +\infty)$ & $(-3, +\infty)$ & $o9_{14062}(\infty) = L(7,q)$ \\ \cline{2-6}
		\hline
		$o9_{34972}$ & $g = 4$ & $\delta = -12$ & $(-6, +\infty)$ & $(-6, +\infty)$ & $o9_{34972}(\infty) = L(7,q)$ \\ \cline{2-6}
		\hline
		$v1076$ & $g = 2$ & $\delta = 4$ & $(-\infty, 2)$ & $(-\infty, 2)$ & $v1076(\infty) = L(9,q)$ \\ \cline{2-6}
		\hline
	$o9_{21619}$ & $g = 3$ & $\delta = -8$ & $(-4, +\infty)$ & $(-4, +\infty)$ & $o9_{21619}(\infty) = L(9,q)$ \\ \cline{2-6}
	\hline
\end{tabular}
\end{table}
\end{ep}

\vspace{8em}

\begin{ep}\rm\label{nonspherical}
	Table \ref{table 4} provides 
	some examples in $\mathcal{N}$ which are complements of type-II knots in some spherical manifolds 
	other than lens spaces, 
	and their obtained CTF filling intervals are equal to their maximal NLS filling intervals.
	
	\begin{table}[h]\label{table 4}
		\caption{Some examples which are complements of type-II knots in spherical manifolds other than lens spaces.}
	\begin{tabular}{| l | l | l | l | l |}
		\hline
		manifold & genus & degeneracy slope & obtained CTF filling interval & maximal NLS filling interval \\
		\hline
		$t08752$ & $g = 2$ & $\delta = -6$ & $(-3,+\infty)$ & $(-3,+\infty)$ \\ \cline{2-5}
		& \multicolumn{4}{| l |}{$t08752(\infty) = W$ for some prism manifold $W$ with $|H_1(W)| = 8$.}  \\ 
		\hline
		$o9_{23699}$ & $g = 2$ & $\delta = -6$ & $(-3,+\infty)$ & $(-3,+\infty)$ \\ \cline{2-5}
		& \multicolumn{4}{| l |}{$o9_{23699}(\infty) = W$ for some tetrahedral manifold $W$ with $|H_1(W)| = 9$.}  \\ 
		\hline
	\end{tabular}
\end{table}
\end{ep}

\begin{ep}\rm\label{nonsharp}
	The CTF filling interval obtained from Corollary \ref{open book} 
	may not be the maximal NLS filling interval for every Floer simple $1$-cusped manifold satisfying
	the assumptions of Corollary \ref{open book}.
	For instance,
	$o9_{19364}$ is an example in $\mathcal{N}$ which
	is the complement of an L-space knot in $S^{3}$ with
	Seifert genus $14$ and degeneracy slope $48$.
	Corollary \ref{open book} gives a CTF filling interval $(-\infty,24)$,
	but the maximal NLS filling interval of $o9_{19364}$ is $(-\infty, 27)$.
\end{ep}

\subsection{Organization}

In Section \ref{section 2}, 
we set up some conventions and review some background material on
surface homeomorphisms,
degeneracy loci of suspension flows,
co-orientable pseudo-Anosov maps,
and branched surfaces.
We prove Theorem \ref{main} in Section \ref{section 3}.
In Subsection \ref{subsection 3.1},
we construct a branched surface $B(\alpha)$ in $M$.
We prove that $B(\alpha)$ is a laminar branched surface in Subsection \ref{subsection 3.2}.
In Subsection \ref{subsection 3.3},
we describe the boundary train tracks of $B(\alpha)$ in $\partial M$ and choose some simple closed curves carried by them.
In Subsection \ref{subsection 3.4},
we prove that the boundary train tracks of $B(\alpha)$ realize all rational multislopes contained in
the multi-interval given in Theorem \ref{main}.
The proof of Theorem \ref{main} is completed in Subsection \ref{proof}.
In Section \ref{section 4},
we prove Propositions \ref{pretzel} and \ref{no interior singularity}.

\section{Preliminaries}\label{section 2}

\subsection{Conventions}\label{subsection 2.1}

For a set $X$,
let $|X|$ denote the cardinality of $X$.
For two metric spaces $A$ and $B$, 
let $A \setminus \setminus B$ denote the closure of $A - B$ under the path metric.

For a link manifold $N$ such that $\partial N$ is a union of tori $\bigcup_{i=1}^{n}T_i$,
a \emph{multislope} on $\partial N$ is 
an $n$-tuple of slopes on
$T_1, \ldots, T_n$ respectively.
For a link in some closed $3$-manifold,
the multislopes on it are defined as
the multislopes on its exterior.

Now we illustrate our conventions of slopes and orientations for
pseudo-Anosov mapping tori of compact orientable surfaces.
Let $\Sigma$ be a compact orientable surface with 
negative Euler characteristic and nonempty boundary,
and
let $\varphi: \Sigma \to \Sigma$ be an orientation-preserving pseudo-Anosov homeomorphism.
Let $M = \Sigma \times I / \stackrel{\varphi}{\sim}$ be the mapping torus of $\Sigma$ over $\varphi$,
and we fix an orientation on $M$.

\begin{notation}\rm
(a)
For two slopes $\alpha, \beta$ in the same boundary component of $M$,
let $\Delta(\alpha, \beta)$ denote the minimal geometric intersection number of $\alpha, \beta$.
	
(b)
For two closed oriented curves $\gamma, \eta$ in the same boundary component of $M$,
let $\langle \gamma, \eta \rangle$ denote the algebraic intersection number of $\gamma, \eta$
(see Convention \ref{orientation} (d) for more details on this setting).
\end{notation}

We adopt the following conventions for slopes on the boundary components of $M$,
as described in
\hyperref[Ro2]{[Ro2]}:

\begin{convention}[Slope conventions]\rm\label{slope}
	Let $T$ be a boundary component of $M$.
	Let $\delta_T$ denote the degeneracy slope of $T$.
	
	(a)
	We call the slope of $T \cap (\Sigma \times \{0\})$ on $T$ the \emph{longitude} of $T$ and
	denote it by $\lambda_T$.
	We choose a slope $\mu_T$ on $T$ such that
	$\Delta(\lambda_T, \mu_T) = 1$ and
	\begin{center}
		$\Delta(\mu_T,\delta_T) \leqslant \Delta(s,\delta_T)$ for any
		slope $s$ of $T$ with $\Delta(\lambda_T, s) = 1$,
	\end{center}
	and we call $\mu_T$ the \emph{meridian} of $T$.
	We fix a canonical orientation on $\lambda_T$ and on $\mu_T$,
	see Convention \ref{orientation} (c) for details.
	Note that $\mu_T$ has a unique choice if $\Delta(\lambda_T, \delta_T) \ne 2$.
	See (c) for the choice of $\mu_T$ in the case of $\Delta(\lambda_T, \delta_T) = 2$.
	
	(b)
	For an essential simple closed curve $\gamma$ on $T$,
	we identify the slope of $\gamma$ with the number
	$$\frac{\langle \gamma, \lambda_T \rangle}{\langle \mu_T, \gamma \rangle} \in \mathbb{Q} \cup \{\infty\}.$$
	
	(c)
	If $\Delta(\lambda_T, \delta_T) = 2$,
	then there are two choices of $\mu_T$ satisfying (a).
	$\delta_T$ is equal to $-2$, $2$ in these two choices respectively.
	We choose $\mu_T$ so that $\delta_T = 2$.
\end{convention}

Next,
we assign orientations to $\Sigma \times \{0\}$ and the meridians and longitudes on $\partial M$.

\begin{convention}[Orientation conventions]\rm\label{orientation}
	(a)
	Note that each orientation on $\Sigma \times \{0\}$ determines a normal vector field
	with respect to the orientation on $M$.
	We orient $\Sigma \times \{0\}$ so that
	the induced normal vector field is consistent with 
	the increasing orientation on the second coordinate in $\Sigma \times I$.
	Accordingly, we assign $\Sigma$ the orientation induced from $\Sigma \times \{0\}$.
	
	(b)
	In our context,
    the left and right sides of any oriented curve in $\Sigma$ will always
    be with respect to the orientation on $\Sigma$.
	Each boundary component of $\Sigma$ has an induced orientation with the following property:
	for any properly embedded arc $\gamma: I \to \Sigma$,
	choose a normal vector field $\{v(x) \mid x \in \gamma(I)\}$ pointing to the right side of $\gamma$,
	then $v(\gamma(0))$ is consistent with the positive orientation on $\partial \Sigma$ and
	$v(\gamma(1))$ is consistent with the negative orientation on $\partial \Sigma$.
	For each component $C$ of $\partial \Sigma$,
	we also orient $C \times \{0\} \subseteq \partial M$ with the orientation induced from $C$.
	
	(c)
	Let $T$ be a component of $\partial M$.
	We assign the longitude of $T$ the orientation consistent with 
	the positive orientation on each component of $T \cap (\partial \Sigma \times \{0\})$.
	We choose a curve $\gamma$ on $T$ such that 
	$\gamma$ represents the meridian on $T$ and $\gamma$ is transverse to the fibered surfaces 
	$\{\Sigma \times \{t\} \mid t \in I\}$,
	and we orient $\gamma$ consistently with the increasing orientation on the second coordinate.
	We assign the meridian on $T$ the orientation consistent with $\gamma$.
	
	(d)
	For a boundary component $T$ of $M$,
	we set $\langle \mu_T, \lambda_T \rangle = - \langle \lambda_T, \mu_T \rangle = 1$,
	where $\mu_T, \lambda_T$ denotes the meridian and longitude on $T$ respectively.
\end{convention}

At last,
we describe the relation between our canonical coordinate system on $\partial M$ and
the standard meridian/longitude coordinate system for knots in $S^{3}$,
when $M$ is the exterior of a fibered knot in $S^{3}$.

\begin{remark}\rm\label{slope consistent}
	Let $M$ be the exterior of a fibered knot $K$ in $S^{3}$,
	and let $T = \partial M$.
	We fix an orientation on $M$.
	Let $(\mu_K, \lambda_K)$ denote the standard coordinate system for $K$ in $S^{3}$,
	where $\mu_K, \lambda_K$ are the standard meridian and longitude of $K$ in $S^{3}$.
    As discussed in \hyperref[Ro2]{[Ro2, Corollary 7.4]}, exactly one of the following two cases occurs:
	\begin{itemize}
		\item[(1)] $\delta_T$ doesn't have slope $-2$ with respect to $(\mu_K, \lambda_K)$, 
		in which case $\mu_K = \mu_T$ and $\lambda_K = \lambda_T$.
		
		\item[(2)] $\delta_T, \mu_T$ have slopes $-2, -1$, respectively, with respect to $(\mu_K, \lambda_K)$.
	\end{itemize}
	
	Now suppose that the case (2) holds. Then $\Delta(\lambda_T, \delta_T) = 2$.  
	The opposite orientation on $M$ induces another canonical coordinate system $(\mu^{'}_{T}, \lambda^{'}_{T})$ on $T$,
	where the longitude $\lambda^{'}_{T}$ coincides with $\lambda_T$ but is assigned the opposite orientation,
	and the meridian $\mu^{'}_{T}$ is distinct from $\mu_T$.
	As illustrated in Convention \ref{slope} (a),  
	we have $\Delta(\mu^{'}_{T}, \lambda^{'}_{T}) = 1$
	and $\Delta(\mu^{'}_{T}, \delta_T) = \Delta(\mu_T, \delta_T) = 1$.
	Since $\delta_T$ has slope $-2$ with respect to $(\mu_K, \lambda_K)$,  
	$\mu^{'}_{T}$ must have slope either $\infty$ or $-1$ with respect to $(\mu_K, \lambda_K)$.  
	Because $\mu_T' \ne \mu_T$, 
	$\mu^{'}_{T}$ can only have slope $\infty$,
	and thus $\mu_K = \mu^{'}_{T}$.
\end{remark}

\subsection{Thurston’s classification of surface homeomorphisms}\label{surface homeomorphism}

Let $\Sigma$ be a compact orientable surface with negative Euler characteristic,
and let $\varphi: \Sigma \to \Sigma$ be an orientation-preserving homeomorphism.
Recall that $\varphi$ is \emph{pseudo-Anosov} if,
there is a pair of measured singular foliations $(\mathcal{F}^{s}, \mu^{s}), (\mathcal{F}^{u}, \mu^{u})$ with
the same set of singularities in $\operatorname{Int}(\Sigma)$,
transverse except at these singularties and at $\partial \Sigma$,
such that
$\varphi$ takes $(\mathcal{F}^{s}, \mu^{s})$ to $(\mathcal{F}^{s}, \lambda^{-1} \mu^{s})$ and
takes $(\mathcal{F}^{u}, \mu^{u})$ to $(\mathcal{F}^{u}, \lambda \mu^{u})$ for 
some real number $\lambda > 1$.
In this case,
$\mathcal{F}^{s}$ (resp. $\mathcal{F}^{u}$) is called 
the \emph{stable foliation} (resp. \emph{unstable foliation}) of $\varphi$.

Here is Thurston's classification of surface homeomorphisms (\hyperref[T1]{[T1]}).
For any orientation-preserving homeomorphism $h: \Sigma \to \Sigma$,
$h$ is freely isotopic to a homeomorphism $h_0$ with one of the following properties:

(1)
$h^{n}_{0} = \text{id}$ for some $n \in \mathbb{N}$.
In this case,
$h_0$ is said to be \emph{periodic}.

(2)
There is a collection of pairwise disjoint essential simple closed curves
$C_1, \ldots, C_k$ in $\Sigma$
such that $h_0(\bigcup_{i=1}^{k}C_i) = \bigcup_{i=1}^{k}C_i$.
In this case,
$h_0$ is said to be \emph{reducible}. 

(3)
$h_0$ is pseudo-Anosov.

In \hyperref[T2]{[T2]},
Thurston proves that
the interiors of pseudo-Anosov mapping tori have
finite volume hyperbolic structures:

\begin{thm}[Thurston]
Let $M = \Sigma \times I / \stackrel{\varphi}{\sim}$ be the mapping torus of $\Sigma$ over $\varphi$.
Then 
$\operatorname{Int}(M)$ has a complete hyperbolic structure of finite volume if and only if
$\varphi$ is isotopic to a pseudo-Anosov homeomorphism.
\end{thm}

\subsection{Degeneracy loci of suspension flows}\label{degeneracy locus}

\begin{figure}\label{singularities}
	\centering
	\includegraphics[width=0.4\textwidth]{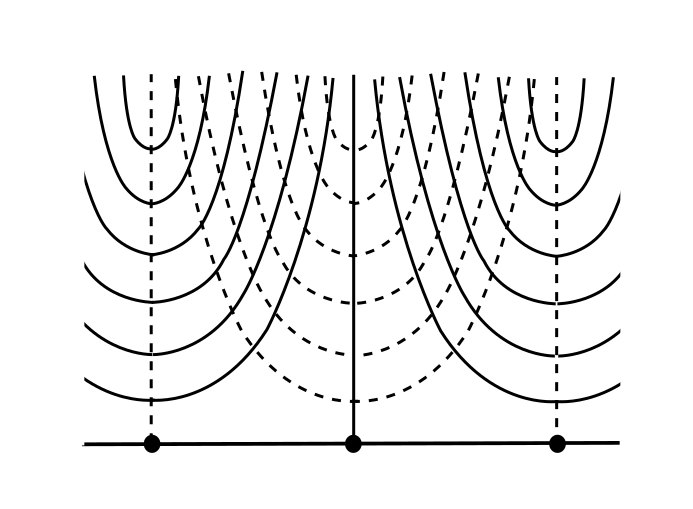}
	\caption{The dots are singularities of $\mathcal{F}^{s}$ or $\mathcal{F}^{u}$ on
		a boundary component of $\Sigma$.
		In $\operatorname{Int}(\Sigma)$,
		the solid lines are leaves of $\mathcal{F}^{s}$, and the dashed lines are leaves of $\mathcal{F}^{u}$.}
\end{figure}

Let $\Sigma$ be a compact orientable surface with negative Euler characteristic,
and let $\varphi: \Sigma \to \Sigma$ be an orientation-preserving pseudo-Anosov homeomorphism with 
the stable and unstable foliations $\mathcal{F}^{s}, \mathcal{F}^{u}$.
We assume further that $\Sigma$ has nonempty boundary.

Let $C$ be a component of $\partial \Sigma$.
Then each singularity of $\mathcal{F}^{s}$ or $\mathcal{F}^{u}$ on $C$ has
three separatrices,
where two of them are contained in $C$ and the other one is a prong toward $\operatorname{Int}(\Sigma)$,
see Figure \ref{singularities}.
Let $k$ be the number of singularities of $\mathcal{F}^{s}$ contained in $C$.
Then there are exactly $2k$ periodic points on $C$, 
consisting of
the $k$ singularities of $\mathcal{F}^{s}$ on $C$ (referred to as the \emph{attracting periodic points}) and
the $k$ singularities of $\mathcal{F}^{u}$ on $C$ (referred to as the \emph{repelling periodic points}).
We refer to \hyperref[CC]{[CC, Appendix]}, \hyperref[Ro2]{[Ro2, page 446]} for
some descriptions.

Let $M = \Sigma \times I / \stackrel{\varphi}{\sim}$ be the mapping torus of $\Sigma$ over $\varphi$,
and let $\Psi$ denote the suspension flow of $\varphi$ in $M$.
Then all closed orbits of $\Psi$ contained in the same boundary component of $M$ are
parallel essential simple closed curves.
Let $T$ be a component of $\partial M$.
Recall that the degeneracy slope of $T$, denoted $\delta_T$, 
is the slope of the closed orbits of $\Psi$ contained in $T$.
Let $\delta_T = \frac{u}{v}$ for which $u, v \in \mathbb{Z}$, $u > 0$, and $\gcd(u, v) = 1$,
where $u = 1, v = 0$ if $\delta_T = \infty$.
As explained above,
there are $2n$ closed orbits of $\Psi$ on $T$ for some $n \in \mathbb{N}_+$,
and $n$ of them contain $(t,0)$ for some $t \in \partial \Sigma$ which is a singular point of $\mathcal{F}^{s}$.
The \emph{degeneracy locus} of $\Psi$ on $T$ is
the union of these $n$ closed orbits,
which is denoted by $d(T)$ and identified with $(nu; nv)$. 
We refer to $n$ as the multiplicity of $d(T)$.
The degeneracy locus of $\Psi$ on $T$ is originally introduced 
as the degeneracy curve of $T$ in \hyperref[GO]{[GO, Section 5]}.
See also \hyperref[Ga4]{[Ga4, Section 8]} and \hyperref[Ro2]{[Ro2]}.

Now we assume further that
$\varphi$ takes each component of $\partial \Sigma$ to itself.
Let $C$ be a component of $\partial \Sigma$, 
and let $T$ denote the component $C \times I$ of $\partial M$.
The \emph{fractional Dehn twist coefficient} of $C$, 
denoted $f_C(\varphi)$, 
is defined to be
$\frac{1}{\delta_T}$
(where $f_C(\varphi) = 0$ if $\delta_T = \infty$),
see \hyperref[KazR1]{[KazR1, page 3]}.
$\varphi$ is called \emph{right-veering} (resp. \emph{left-veering}) if
all components of $\partial \Sigma$ have positive (resp. negative) fractional Dehn twist coefficients,
see \hyperref[HKM]{[HKM, Proposition 3.1]}.
We note that $f_C(\varphi) > 0$ (resp. $f_C(\varphi) < 0$) if and only if $\delta_T > 0$ (resp. $\delta_T < 0$).
Thus,
$\varphi$ is right-veering (resp. left-veering) if and only if
all components of $\partial \Sigma$ have positive (resp. negative) degeneracy slopes.

\subsection{Co-orientable pseudo-Anosov maps}\label{co-orientable}

Let $\Sigma$ be a compact orientable surface with negative Euler characteristic,
and let $\varphi: \Sigma \to \Sigma$ be an orientation-preserving pseudo-Anosov homeomorphism with 
the stable and unstable foliations $\mathcal{F}^{s}, \mathcal{F}^{u}$.

Recall that $\varphi$ is a co-orientable map if $\mathcal{F}^{s}$ is co-orientable,
and $\varphi$ is co-orientation-preserving (or co-orientation-reversing) if
$\varphi$ is co-orientable and preserves (or reverses) the co-orientation on $\mathcal{F}^{s}$.
We first explain that
$\mathcal{F}^{s}$ can be replaced with $\mathcal{F}^{u}$ above:

\begin{remark}\rm\label{independent}
If $\mathcal{F}^{s}$ is co-oriented,
	then the co-orientation on $\mathcal{F}^{s}$ defines 
	continuously varying orientations on the leaves of $\mathcal{F}^{u}$,
	and thus defines a co-orientation on $\mathcal{F}^{u}$.
	Therefore,
	$\mathcal{F}^{s}$ is co-orientable if and only if $\mathcal{F}^{u}$ is co-orientable,
	and the co-orientations on $\mathcal{F}^{s}, \mathcal{F}^{u}$ are either
	simultaneously preserved or simultaneously reversed by $\varphi$.
\end{remark}

We refer to each singularity of $\mathcal{F}^{s}$ contained in $\operatorname{Int}(\Sigma)$ as
an \emph{interior singularity} of $\mathcal{F}^{s}$.
If $\varphi$ is co-orientable,
then every interior singularity of $\mathcal{F}^{s}$ 
has an even number of prongs, 
and every component of $\partial \Sigma$ contains an even number of singularities of $\mathcal{F}^{s}$.

In the case where $\varphi$ is non-co-orientable,
$\Sigma$ has a double covering branched over 
the set of interior singularities of $\mathcal{F}^{s}$ with an odd number of prongs,
denoted $p: \widetilde{\Sigma} \to \Sigma$, 
such that $\mathcal{F}^{s}$ pulls back to 
a co-orientable singular foliation $\widetilde{\mathcal{F}^{s}}$ of $\widetilde{\Sigma}$.
Then $\mathcal{F}^{u}$ also pulls back to
a co-orientable singular foliation $\widetilde{\mathcal{F}^{u}}$ of $\widetilde{\Sigma}$,
and
the measures on $\mathcal{F}^{s}, \mathcal{F}^{u}$ pull back to
two measures on $\widetilde{\mathcal{F}^{s}}, \widetilde{\mathcal{F}^{u}}$.
We refer the reader to 
\hyperref[BB]{[BB, Section 2.4]},
\hyperref[FM]{[FM, 14.2.1]} for details.
$\widetilde{\Sigma}$ also has the following properties:

(1)
The branched covering $p: \widetilde{\Sigma} \to \Sigma$ has a deck group $\mathbb{Z} / 2\mathbb{Z}$,
consisting of all homeomorphisms 
$f: \widetilde{\Sigma} \to \widetilde{\Sigma}$ with $p = p \circ f$.
The generator of this deck group,
denoted $\theta$,
reverses the co-orientation on $\widetilde{\mathcal{F}^{s}}$.

(2)
$\varphi$ lifts to a homeomorphism
$\widetilde{\varphi}: \widetilde{\Sigma} \to \widetilde{\Sigma}$.
Since $\widetilde{\varphi}$ preserves $\widetilde{\mathcal{F}^{s}}, \widetilde{\mathcal{F}^{u}}$,
scaling their measures by the same factor as $\varphi$,
it follows that $\widetilde{\varphi}$ is a pseudo-Anosov map.

Note that 
$\theta \circ \widetilde{\varphi}: \widetilde{\Sigma} \to \widetilde{\Sigma}$ is also a lift of $\varphi$.
Similar to (2),
$\theta \circ \widetilde{\varphi}$ is also a pseudo-Anosov map.
Clearly,
$\widetilde{\varphi}, \theta \circ \widetilde{\varphi}$ are co-orientable maps.
As $\theta$ reverses the co-orientation on $\widetilde{\mathcal{F}^{s}}$,
either $\widetilde{\varphi}$ or $\theta \circ \widetilde{\varphi}$ is co-orientation-preserving, 
and the other is co-orientation-reversing.

We refer to \hyperref[T1]{[T1]}, \hyperref[M]{[M]} for 
more information on co-orientable pseudo-Anosov homeomorphisms, 
and refer to \hyperref[BB]{[BB, Lemma 4.3]},
\hyperref[DGT]{[DGT]} for criteria to determine whether 
a pseudo-Anosov map is co-orientation-preserving/reversing.

\subsection{Branched surfaces}\label{subsection 2.2}

\begin{figure}\label{standard spine picture}
	\centering
	\subfigure[]{
	\includegraphics[width=0.25\textwidth]{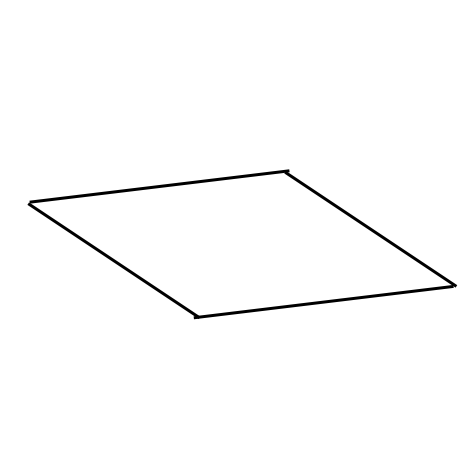}}
\subfigure[]{
	\includegraphics[width=0.25\textwidth]{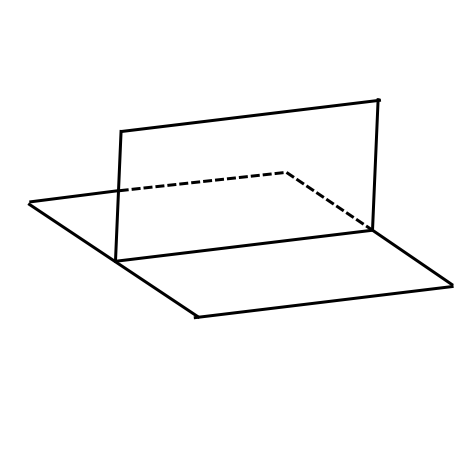}}
	\subfigure[]{
	\includegraphics[width=0.25\textwidth]{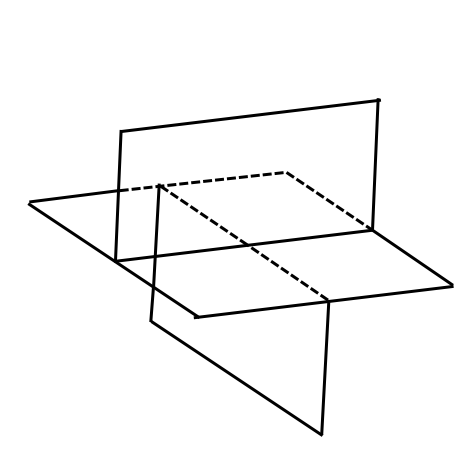}}
	\subfigure[]{
		\includegraphics[width=0.3\textwidth]{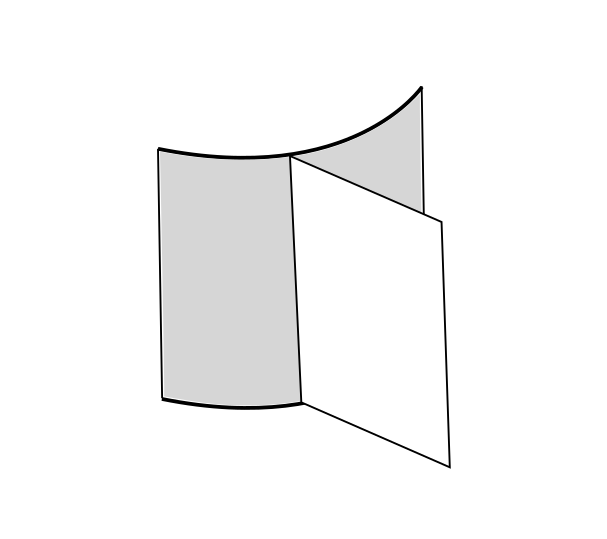}}
	\subfigure[]{
		\includegraphics[width=0.3\textwidth]{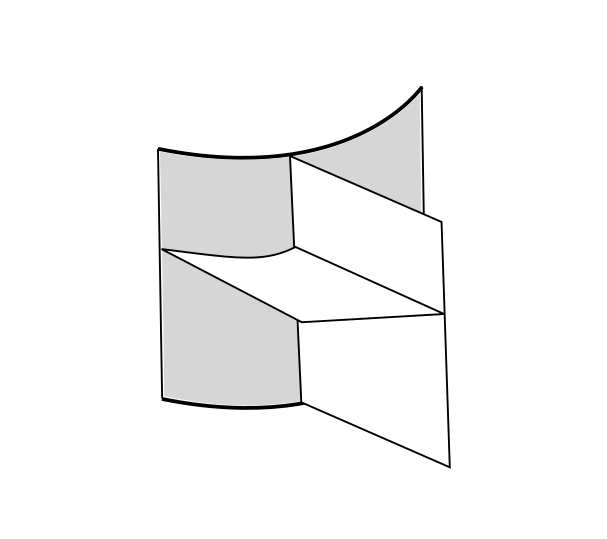}}
	\caption{Local models of standard spines, where the shaded regions in (d) and (e) are contained in $\partial M$.}
\end{figure}

Branched surfaces are an important tool to describe foliations and laminations.
In this subsection,
we review some background on branched surfaces.
We refer to \hyperref[FO]{[FO]}, \hyperref[O2]{[O2]} for some expositions on branched surfaces.

\begin{defn}\rm \label{standard spine}
	A \emph{standard spine} $X$ in a compact, orientable $3$-manifold $M$
	is an embedded $2$-complex such that 
	every point of $X - \partial M$ has a neighborhood in $X$ modeled on 
	one of Figures~\ref{standard spine picture} (a)$\sim$(c), 
	and every point of $X \cap \partial M$ has a neighborhood in $X$ modeled on 
	Figure~\ref{standard spine picture} (d) or~(e).
\end{defn}

\begin{defn}\rm \label{branched surface definition}
	A \emph{branched surface} $B$ in a compact orientable $3$-manifold $M$ is
	a standard spine with a well-defined tangent plane at every point.
For every point $x \in B$, 
$x$ has a neighborhood in $B$ modeled on one of Figures \ref{branched surface} (a)$\sim$(c) if $x \in B - \partial M$, 
and a neighborhood in $B$ modeled on Figure \ref{branched surface} (d) or (e) if $x \in B \cap \partial M$.
\end{defn}

Let $B$ be a branched surface in a compact orientable $3$-manifold $M$,
and let $$L(B) = \{t \in B \mid t \text{ has no Euclidean neighborhood in } B\}.$$ 
Call $L(B)$ the \emph{branch locus} of $B$,
call each component of $B \setminus \setminus L(B)$ a \emph{branch sector},
call each point in $L(B)$ without an $\mathbb{R}$-neighborhood in $L(B)$
a \emph{double point} of $L(B)$,
and call each component of
$L(B) \setminus \setminus \{\text{double points in } L(B)\}$
a \emph{segment} in $L(B)$.
$B$ is \emph{co-orientable} if there exists 
a continuous normal vector field on $B$,
or equivalently,
all branch sectors of $B$ have compatible co-orientations.

A \emph{fibered neighborhood} $N(B)$ of $B$ is
a regular neighborhood of $B$ locally modeled as Figure \ref{fibered neighborhood}.
We regard $N(B)$ as an interval bundle over $B$
and call its fibers \emph{interval fibers}.
$\partial N(B) \setminus \setminus \partial M$ can be decomposed into 
two (possibly disconnected) compact subsurfaces
$\partial_h N(B)$ and $\partial_v N(B)$,
such that $\partial_h N(B)$ is transverse to the interval fibers and
$\partial_v N(B)$ is tangent to the interval fibers.
We call $\partial_h N(B)$ the \emph{horizontal boundary} of $N(B)$ and
call $\partial_v N(B)$ the \emph{vertical boundary} of $N(B)$.
Let $\pi: N(B) \to B$ denote the canonical projection that
sends every interval fiber to a single point.
We call $\pi$ the \emph{collapsing map} for $N(B)$.
Throughout this paper,
the notation $\operatorname{Int}(N(B))$ is always assumed to mean $N(B) - (\partial_h N(B) \cup \partial_v N(B))$,
so that the space $M - \operatorname{Int}(N(B))$ is identified with the closure of $M - N(B)$.

\begin{figure}\label{branched surface}
	\centering
	\subfigure[]{
		\includegraphics[width=0.25\textwidth]{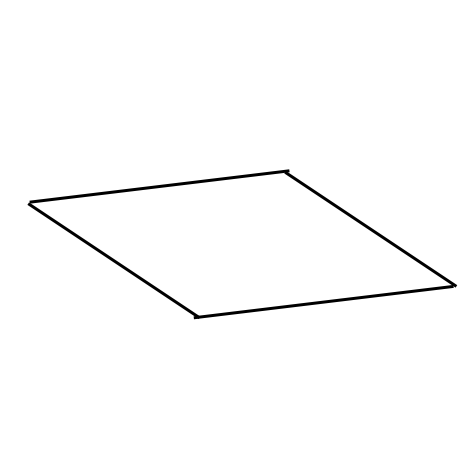}}
	\subfigure[]{
		\includegraphics[width=0.25\textwidth]{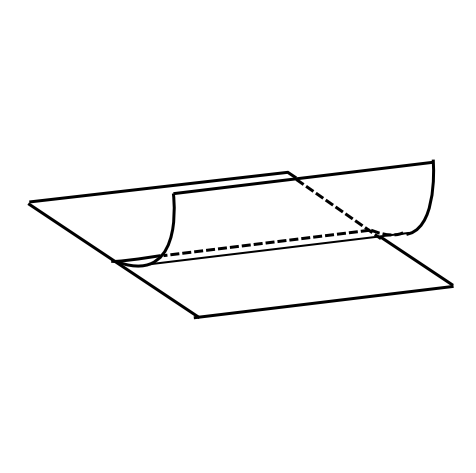}}
	\subfigure[]{
		\includegraphics[width=0.25\textwidth]{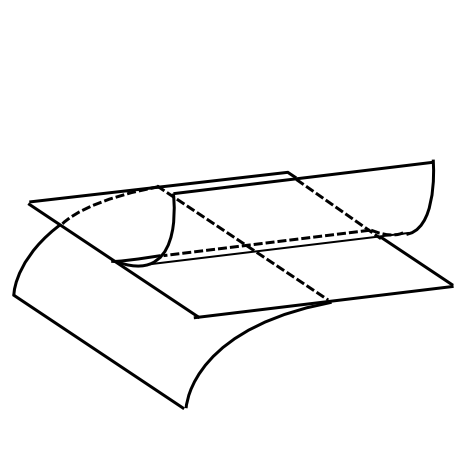}}
	\subfigure[]{
		\includegraphics[width=0.3\textwidth]{30.png}}
	\subfigure[]{
		\includegraphics[width=0.3\textwidth]{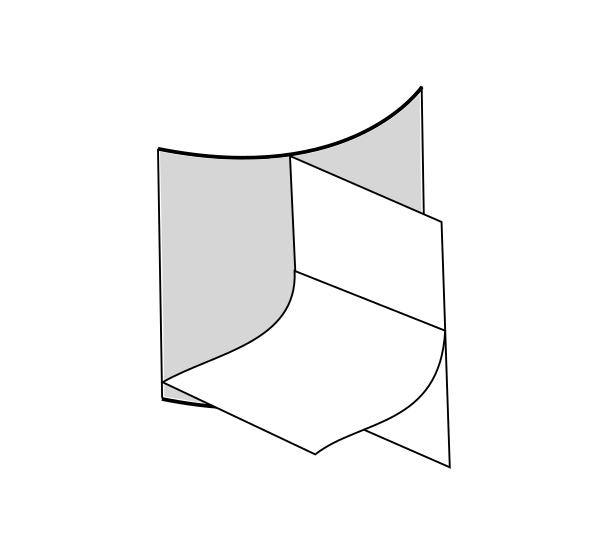}}
	\caption{Local models of branched surfaces, where the shaded regions in (d) and (e) are contained in $\partial M$.}
\end{figure}

We note that $\pi$ projects $\partial_v N(B)$ onto $L(B)$.
For each segment $s$ in $L(B)$, 
there exists a unique component $V$ of $\partial_v N(B)$ such that $s \subseteq \pi(V)$. 
We assign $s$ a normal vector within $B$, induced by the inward-pointing normal vector on $V$ 
(i.e. the normal vector on $V$ pointing into $N(B)$), 
which is referred to as the \emph{cusp direction} at $s$.

\begin{defn}\rm
A lamination $\mathcal{L}$ is \emph{carried} by $B$
if we can choose a fibered neighborhood $N(B)$ of $B$
such that 
$\mathcal{L} \subseteq N(B)$ and
every leaf of $\mathcal{L}$ is transverse to the interval fibers of $N(B)$.
Furthermore, $\mathcal{L}$ is \emph{fully carried} by $B$ if
$\mathcal{L}$ is carried by $B$ and intersects
all interval fibers of $N(B)$.
\end{defn}

In \hyperref[GO]{[GO]},
Gabai and Oertel introduce essential branched surfaces to describe essential laminations.

\begin{defn}[Essential branched surface]\rm \label{essential}
	A branched surface $B$ in a compact orientable $3$-manifold $M$ is \emph{essential} if all of 
	the following conditions hold:
	
	(a)
	There is no disk of contact,
	where a disk of contact is an embedded disk $D \subseteq N(B)$ transverse to 
	the interval fibers of $N(B)$ with $\partial D \subseteq \operatorname{Int}(\partial_v N(B))$.
	In addition, there is no half disk of contact,
	where a half disk of contact is an embedded disk $D \subseteq N(B)$ transverse to 
	the interval fibers such that
	there is a connected segment $\alpha \subseteq \partial D$ with
	$\alpha \subseteq \partial M \cap \partial N(B)$ and
	$\partial D - \operatorname{Int}(\alpha) \subseteq \operatorname{Int}(\partial_v N(B)) - \partial M$.
	
	(b)
	$\partial_h N(B)$ is incompressible and $\partial$-incompressible in $M -\operatorname{Int}(N(B))$,
	no component of $\partial_h N(B)$ is a sphere or a disk properly embedded in $M$,
	and no monogon exists in $M - \operatorname{Int}(N(B))$.
	Here a \emph{monogon} in $M - \operatorname{Int}(N(B))$ means
	an embedded disk $D \subseteq M - \operatorname{Int}(N(B))$ such that 
	$D \cap N(B) = \partial D$, 
	and $\partial D \cap \partial_v N(B)$ lies in a single interval fiber of $N(B)$.
	
	(c)
	$M - \operatorname{Int}(N(B))$ is irreducible,
	and $\partial M - \operatorname{Int}(N(B))$ is incompressible in $ M - \operatorname{Int}(N(B))$.
	
	(d)
	$B$ contains no Reeb branched surface 
	(see \hyperref[GO]{[GO]} for the definition of Reeb branched surface).
	
	(e)
	$B$ fully carries a lamination.
\end{defn}

\begin{figure}\label{fibered neighborhood}
	\includegraphics[width=0.7\textwidth]{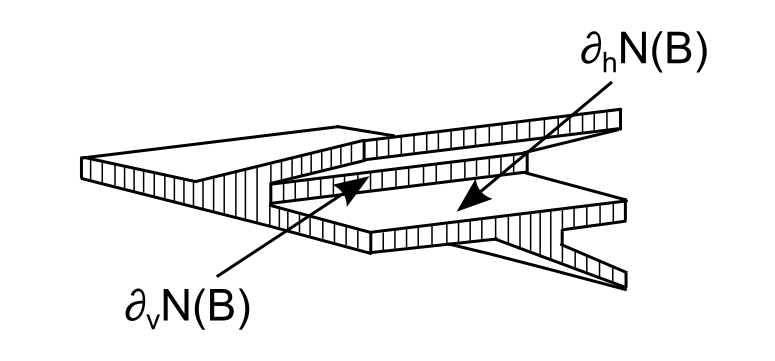}
	\caption{The fibered neighborhood $N(B)$.}
\end{figure}

Gabai and Oertel (\hyperref[GO]{[GO]}) prove that 

\begin{thm}[Gabai-Oertel]
	(a)
	Any essential lamination in a compact orientable $3$-manifold is
	fully carried by an essential branched surface.
	
	(b)
	Any lamination in a compact orientable $3$-manifold fully carried by
	an essential branched surface is an essential lamination.
\end{thm}

Now we describe the laminar branched surface introduced by Li in \hyperref[Li1]{[Li1]}, \hyperref[Li2]{[Li2]}.

\begin{defn}\rm
	Let $B$ be a branched surface.
	
	(a)
	A \emph{sink disk} of $B$ is 
	a disk component $D$ of $B \setminus \setminus L(B)$ such that
	$D \cap \partial M = \emptyset$ and 
	the cusp directions at all segments in $\partial D$ point into $D$.
	
	(b)
	A \emph{half sink disk} of $B$ is a disk component $D$ of $B \setminus \setminus L(B)$ such that
	$D \cap \partial M \ne \emptyset$ and 
	the cusp directions at all segments in $\partial D \cap L(B)$ point into $D$.
\end{defn}

For a branched surface $B$ in a $3$-manifold $M$,
A \emph{trivial bubble} in $M - \operatorname{Int}(N(B))$ is 
a $3$-ball component $Q$ of $M - \operatorname{Int}(N(B))$ such that
$\partial Q \cap \partial_h N(B)$ has $2$ components and each of them is a disk, 
and moreover,
$\pi \mid_{\partial Q \cap N_h(B)}$ is injective.
We also call $\pi(\partial Q)$ a \emph{trivial bubble} of $B$.

\begin{defn}[Laminar branched surface]\rm\label{laminar}
	Let $B$ be a branched surface in a compact orientable $3$-manifold $M$.
	$B$ is a \emph{laminar branched surface} if $B$ satisfies Conditions
	(b)$\sim$(d) of Definition \ref{essential},
	$B$ contains no trivial bubble,
	and $B$ contains no sink disk or half sink disk.
\end{defn}

In \hyperref[Li1]{[Li1]}, \hyperref[Li2]{[Li2]},
Li proves:

\begin{thm}[Li]
	Let $M$ be a compact orientable $3$-manifold.
	
	(a)
	Any laminar branched surface in $M$ fully carries an essential lamination.
	
	(b)
	Any essential lamination in $M$ which is not a lamination by $2$-planes is fully carried by a laminar branched surface.
\end{thm}

The concept of fibered neighborhoods for branched surfaces naturally extends to train tracks.
Accordingly, it's also natural to describe curves as being (fully) carried by train tracks in a similar sense.
Let $\tau$ be a train track on a surface $S$.
A fibered neighborhood $N(\tau)$ of $\tau$ is a regular neighborhood of $\tau$,
foliated by closed intervals (still referred to as the \emph{interval fibers} of $N(\tau)$),
such that there is a well-defined quotient map $q: N(\tau) \to \tau$
that maps each interval fiber to a single point of $\tau$.
A one-dimensional lamination $\Lambda$ on $S$ is said to be carried by $\tau$ if
there exists a fibered neighborhood of $\tau$ such that $\Lambda$ is transverse to its interval fibers.
In addition, $\Lambda$ is \emph{fully carried} by $\tau$ if $\Lambda$ is carried by $\tau$ and 
intersects every interval fiber of $N(\tau)$.

Now, consider a compact $3$-manifold $M$ such that $\partial M$ contains a torus component $T$,
and let $\tau$ be a train track on $T$.
Following \hyperref[Li2]{[Li2, page 4]},
a slope $s$ on $T$ is realized by $\tau$ if $\tau$ fully carries a union of simple closed curves with slope $s$.

The following theorem demonstrates that 
the laminar branched surface is a useful tool
for constructing laminations and foliations in the Dehn fillings,
by realizing the corresponding slopes from the boundary train tracks (\hyperref[Li2]{[Li2, Theorem 2.2]}):

\begin{thm}[Li]\label{boundary}
	Let $M$ be a compact, orientable, irreducible $3$-manifold whose boundary is a union of incompressible tori,
	and let $T_1, \ldots, T_n$ denote the components of $\partial M$.
	Let $B$ be a laminar branched surface in $M$ such that $\partial M \setminus \setminus B$ is a union of bigons.
	Suppose that $$\textbf{s} = (s_1,\ldots,s_n) \in (\mathbb{Q} \cup \{\infty\})^{n}$$
	is a multislope in $\partial M$ such that
	$s_i$ is a slope on $T_i$ that can be realized by the boundary train track $B \cap T_i$ and
	$B$ does not carry a torus that bounds a solid torus in $M(\textbf{s})$,
	then $B$ fully carries an essential lamination $\mathcal{L}_{\textbf{s}}$ 
	that meets each $T_i$ transversely in a collection of simple closed curves of slope $s_i$.
	Moreover,
	$\mathcal{L}_{\textbf{s}}$ can be extended to an essential lamination in 
	the Dehn filling of $M$ along $\partial M$ with the multislope $\textbf{s}$.
\end{thm}

\begin{remark}\rm
	In \hyperref[Li2]{[Li2, Theorem 2.2]},
	Li only states this theorem in the case that $M$ has connected boundary.
	As noted in \hyperref[KalR]{[KalR, Subsection 2.4]},
	the argument holds for 
	the case that $M$ has multiple boundary components.
\end{remark}

\section{Proof of the main theorem}\label{section 3}

We prove Theorem \ref{main} in this section.

Let $\Sigma$ be a compact orientable surface with negative Euler characteristic and nonempty boundary,
and
let be $\varphi: \Sigma \to \Sigma$ be an orientation-preserving pseudo-Anosov homeomorphism.
Let $M = \Sigma \times I / \stackrel{\varphi}{\sim}$ be the mapping torus of $\Sigma$ over $\varphi$.
We fix an orientation on $M$,
and we adopt Conventions \ref{slope}, \ref{orientation} for $M$.
We assume 

\begin{assume}\rm
	$\varphi$ is co-orientable and co-orientation-reversing.
\end{assume}

Let $\mathcal{F}^{s}, \mathcal{F}^{u}$ denote the stable and unstable foliations of $\varphi$.
Then $\mathcal{F}^{s}, \mathcal{F}^{u}$ are co-orientable.
We fix a co-orientation on $\mathcal{F}^{s}$.

\subsection{Construction of the branched surface}\label{subsection 3.1}

In this subsection,
we construct a branched surface in $M$.
We first choose a union of oriented properly embedded arcs $\alpha \subseteq \Sigma$ and
then construct a branched surface $B(\alpha)$ by
adding a union of product disks $\alpha \times I$ to the fibered surface $\Sigma \times \{0\}$.

Let $C$ be a boundary component of $\Sigma$.
Let $p$ denote the number of singularities of $\mathcal{F}^{s}$ contained in $C$,
and let $v_1, v_2, \ldots, v_p \in C$ denote these $p$ singularities
(consecutive along the positive orientation on $C$).
Note that $2 \mid p$ since $\mathcal{F}^{s}$ is co-orientable.
Cutting open $C$ along $\{v_1, \ldots, v_p\}$,
the resulting space $C \setminus \setminus \{v_1, \ldots, v_p\}$ is a collection of $p$ closed segments,
where we denote by $[v_i, v_{i+1}]$ the segment with endpoints $v_i, v_{i+1}$ and
call it a \emph{stable segment} of $C$,
for each $i \in \{1,\ldots,p\}$ (by convention, $v_{p+1} = v_1$).
For each stable segment $[v_i, v_{i+1}]$,
we choose a point $w_i$ in its interior and
choose a transversal $\tau_i: I \to \Sigma$ of $\mathcal{F}^{s}$ that starts at $w_i$.
Then we call $[v_i, v_{i+1}]$ a \emph{negative} (resp. \emph{positive}) stable segment if
$\tau_i$ is a positively oriented (resp. negatively oriented) transversal of $\mathcal{F}^{s}$.
Note that a stable segment being positive (resp. negative) implies that
its two adjacent stable segments are negative (resp. positive).

\begin{figure}\label{smoothing}
	\centering
		\includegraphics[width=0.3\textwidth]{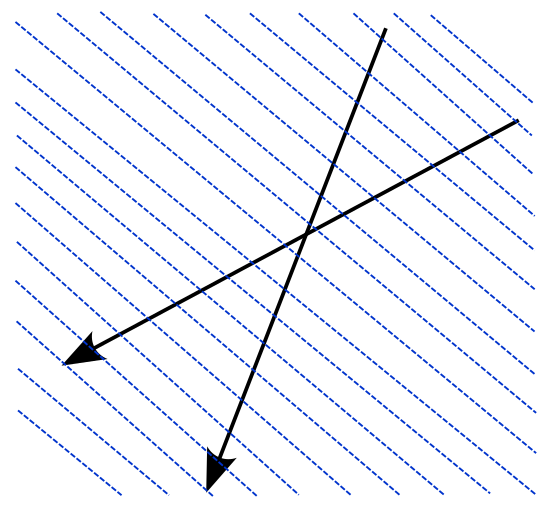}
		\includegraphics[width=0.3\textwidth]{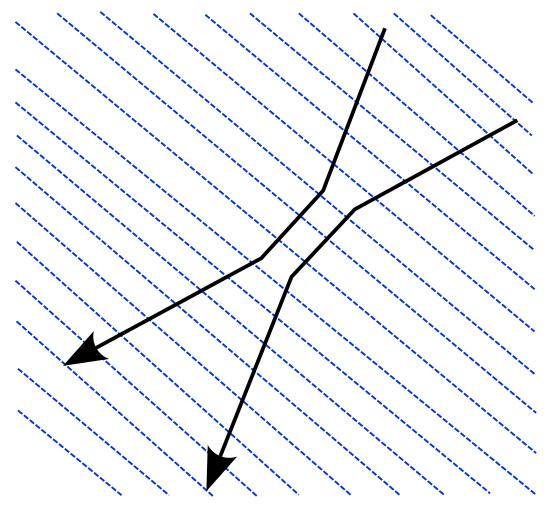}
	\caption{At a double intersection point of a union of paths positively transverse to $\mathcal{F}^{s}$,
	we smooth it with respect to the orientations on the paths.
    Then the paths are still positively transverse to $\mathcal{F}^{s}$.}
\end{figure}

\begin{construction}\rm\label{arcs}
	Let $$A_+ = \{\text{positive stable segments in } \partial \Sigma\},$$
	$$A_- = \{\text{negative stable segments in } \partial \Sigma\}.$$
	Note that $|A_+| = |A_-|$.
	We choose a bijection $j: A_- \to A_+$.
	For each $\sigma \in A_-$,
	we choose a path $r_\sigma$ that 
	starts at some point in $\operatorname{Int}(\sigma)$,
	positively transverse to $\mathcal{F}^{s}$ and disjoint from the singularities of $\mathcal{F}^{s}$,
	and ends at some point in $\operatorname{Int}(j(\sigma))$,
	where the procedure is as follows:
	
	$\bullet$
	We draw a positively oriented transversal $\tau_1: I \to \Sigma$ of $\mathcal{F}^{s}$ that
	starts at some point in $\operatorname{Int}(\sigma)$,
	and we draw a negatively oriented transversal $\tau_2: I \to \Sigma$ of $\mathcal{F}^{s}$ that
	starts at some point in $\operatorname{Int}(j(\sigma))$.
	We assume that both $\tau_1$ and $\tau_2$ are disjoint from the singularities of $\mathcal{F}^{s}$.
	Recall from \hyperref[FM]{[FM, Corollary 14.15]},
	every leaf of $\mathcal{F}^{s}$ that is not contained in $\partial \Sigma$ is dense in $\Sigma$.
	So there exists a non-singular leaf $\lambda$ of $\mathcal{F}^{s}$ and $t_1, t_2 \in I$ such that
	$\tau_1(t_1), \tau_2(t_2) \in \lambda$.
	Let $r^{''}_{\sigma}$ denote 
	the immersed path that starts at $\tau_1(0)$ and goes along $\tau_1([0,t_1])$ to $\tau_1(t_1)$,
	then goes along $\lambda$ to $\tau_2(t_2)$,
	finally goes along the inverse direction of $\tau_2([0,t_2])$ and ends at $\tau_2(0)$.
	Then $r^{''}_{\sigma}$ is positively transverse to $\mathcal{F}^{s}$ except for
	a subarc contained in $\lambda$.
	Let $\eta$ denote this subarc.
	Since $\lambda$ is a non-singular leaf and $\eta$ is compact,
	$\eta$ has a product neighborhood $\eta \times I$ in $\mathcal{F}^{s}$.
	We can isotope $r^{''}_{\sigma}$ in $\eta \times I$ to make it transverse to $\mathcal{F}^{s}$ and
	keep it still disjoint from the singularities of $\mathcal{F}^{s}$.
	Then we obtain an immersed path $r^{'}_{\sigma}$ which 
	is positively transverse to $\mathcal{F}^{s}$.
	We smooth every self-intersection point of $r^{'}_{\sigma}$ with respect to
	the orientation on $r^{'}_{\sigma}$ (Figure \ref{smoothing}).
	This operation produces an oriented properly embedded arc and a (possibly empty) union of oriented circles,
	and we remove all these circles.
	Let $r_\sigma$ denote the remained oriented properly embedded arc.
	
	We can construct $\{r_\sigma \mid \sigma \in A_-\}$ one-by-one to
	make them only have double intersection points
	(this can be guaranteed since each $r_\sigma$ is disjoint from the singularities of $\mathcal{F}^{s}$) 
	and satisfy that 
	$\varphi$ does not take any endpoint in $\{r_\sigma \mid \sigma \in A_-\}$ to
	another endpoint in it.
	We smooth every double intersection point of $\bigcup_{\sigma \in A_-} r_\sigma$ 
	(which is a non-singular point of $\mathcal{F}^{s}$) with respect to the orientations on the paths
	(Figure \ref{smoothing})
	to obtain a finite union of disjoint oriented properly embedded arcs
	which are positively transverse to $\mathcal{F}^{s}$.
	Let $\alpha$ denote this union of oriented arcs.
\end{construction}

We note that

\begin{fact}\rm\label{endpoint}
	Each negative stable segment contains
	exactly one starting endpoint of some path in $\alpha$ and
	no ending endpoint,
	and each positive stable segment contains
	exactly one ending endpoint of some path in $\alpha$ and
	no starting endpoint.
\end{fact}

Let $$S(\alpha) = (\Sigma \times \{0\}) \cup (\alpha \times I) / \stackrel{\varphi}{\sim}$$ be a standard spine in $M$,
which is the image of the subspace $(\Sigma \times \{0\}) \cup (\alpha \times I) \subseteq \Sigma \times I$ 
under the quotient map $\Sigma \times I \to M$.
In the following statements,
whenever we mention subspaces of $\Sigma \times I$,
we will always regard them as their images in $M$.
We now construct a branched surface $B(\alpha)$ as follows.

\begin{defn}\rm\label{B(alpha)}
	(a)
	We assign the singular set of $S(\alpha)$ cusp directions such that,
	for each component $\gamma$ of $\alpha$,
	the cusp direction at $\gamma \times \{0\}$ points to its left side in $\Sigma \times \{0\}$,
	and the cusp direction at $\gamma \times \{1\} = \varphi(\gamma) \times \{0\}$
	points to its right side in $\Sigma \times \{0\}$.
	Let $B(\alpha)$ denote the resulting branched surface.
	
	(b)
	For each component $\gamma$ of $\alpha$,
	we call $\gamma \times I$ a \emph{product disk} and 
	call $\gamma \times \{0\}$ (resp. $\gamma \times \{1\}$) the \emph{lower arc}
	(resp. \emph{upper arc}) of this product disk.
	
	(c)
	Note that both $\alpha$ and $\varphi(\alpha)$ contain no singularities of $\mathcal{F}^{s}$,
	and that $\partial \alpha \cap \varphi(\partial \alpha) = \emptyset$ since $\varphi$ is co-orientation-reversing.
	There is a union of oriented properly embedded arcs $\beta$ in $\Sigma$ such that,
	$\beta$ is isotopic to $\varphi(\alpha)$ relative to the endpoints,
	$\beta$ is transverse to $\mathcal{F}^{s}$,
	and $\alpha,\beta$ only have double intersection points.
	We isotope $\alpha \times I$ relative to 
	$(\alpha \times \{0\}) \cup (\partial \alpha \times I)$ so that
	the upper arcs of $\alpha \times I$ are isotoped to $\beta \times \{0\}$.
	This makes $B(\alpha)$ locally modeled as in Figure \ref{branched surface}.
\end{defn}

As illustrated in Figure \ref{co-orientation},
the fibered surface $\Sigma \times \{0\}$ and
all product disks of $B(\alpha)$ admit compatible co-orientations.
Thus
$B(\alpha)$ is a co-orientable branched surface.
In addition,
$N(B(\alpha))$ has product complementary regions.

\begin{figure}\label{co-orientation}
	\centering
	\includegraphics[width=0.5\textwidth]{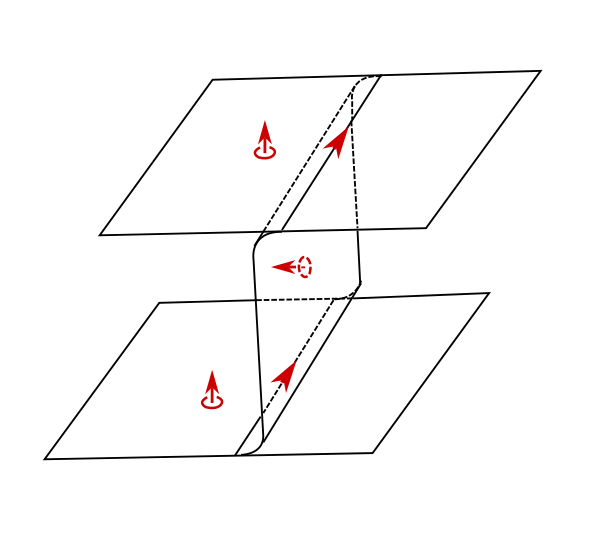}
	\caption{In this picture, the horizontal sectors lie in $\Sigma \times \{0\}$,
		and the vertical sector is contained in a product disk of $B(\alpha)$.  
		The arrows along the branch locus indicate the orientation induced by $\alpha$.
		As shown by the arrows on the sectors (each representing a co-orientation),  
		all sectors can be assigned compatible co-orientations.}
\end{figure}

\subsection{Verifying that $B(\alpha)$ is laminar}\label{subsection 3.2}
In this subsection,
we verify that $B(\alpha)$ is a laminar branched surface.

\begin{lm}\label{transverse}
	Let $\rho$ be an oriented simple closed curve or an oriented properly embedded arc in $\Sigma$ such that 
	$\rho$ can be divided into finitely many segments which 
	are either positively transverse to $\mathcal{F}^{s}$ or
	tangent to $\mathcal{F}^{s}$,
	and at least one of these segments is positively transverse to $\mathcal{F}^{s}$.
	Then $\rho$ is essential in $\Sigma$.
\end{lm}

We offer two proofs of this lemma.
A proof is presented below, and an alternative proof,
based on the structure of the leaf space of $\mathcal{F}^{s}$, 
is given in the Appendix \ref{appendix}.

\begin{proof}
	We first consider the case where $\rho$ is an oriented simple closed curve.
	We assume that $\rho$ is non-essential.
	Then there is an embedded disk $D \subseteq \Sigma$ such that $\rho = \partial D$.

	We refer to each leaf of $\mathcal{F}^{s}$ without singularities as a \emph{regular leaf} of $\mathcal{F}^{s}$.
	Since every leaf of $\mathcal{F}^{s}$ that intersects $\partial \Sigma$ contains some singularity on $\partial \Sigma$,
	every regular leaf of $\mathcal{F}^{s}$ must be contained in $\operatorname{Int}(\Sigma)$.
	We now choose a regular leaf of $\mathcal{F}^{s}$ that has only transverse intersections with $\rho$, as follows.
	Note that $\mathcal{F}^{s}$ has finitely many singular leaves and uncountably many regular leaves,
	and only finitely many leaves of $\mathcal{F}^{s}$ have tangent intersections with $\rho$.
	Thus, there exists a regular leaf $\lambda$ of $\mathcal{F}^{s}$ that has no tangent intersection with $\rho$.
	Since $\lambda$ is dense in $\Sigma$ (\hyperref[FM]{[FM, Corollary 14.15]}),
	$\lambda$ must intersect $D$,
	and $\lambda$ is transverse to $\rho$ at each intersection $\lambda \cap \rho$.

	We note that $\mathcal{F}^{s}$ can be split open along the singular leaves,
	to obtain a lamination carried by some train track $\tau^{s}$.
	Since each component of $\tau^{s} \cap D$ is compact,
	we can ensure that each component of $\lambda \cap D$ is also compact.
	As $\lambda$ is a regular leaf of $\mathcal{F}^{s}$,
	$\lambda \cap D$ can only be a union of closed segments.
	Let $s$ be a component of $\lambda \cap D$.
	Then $\rho$ is positively transverse to $\lambda$ at one endpoint of $s$ and is
	negatively transverse to $\lambda$ at the other endpoint of $s$
	(Figure \ref{intersection}).
	This is a contradiction.
	
	\begin{figure}\label{intersection}
		\centering
		\includegraphics[width=0.35\textwidth]{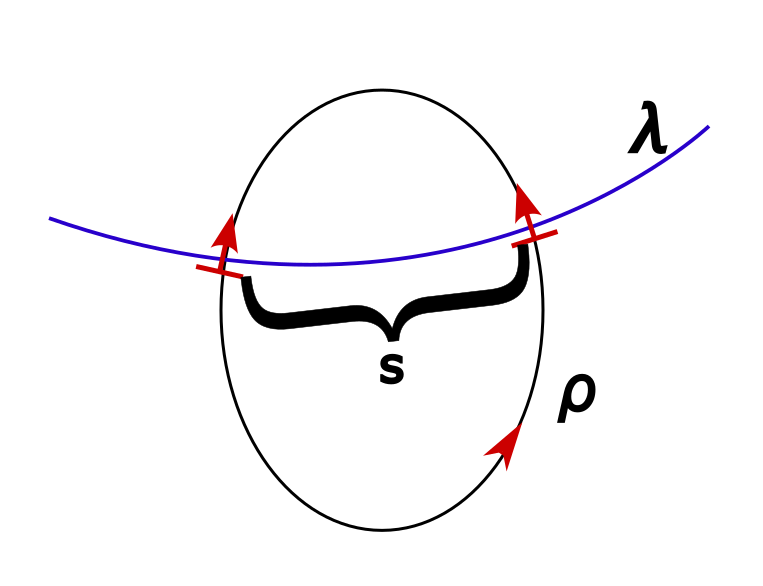}
		\caption{For a component $s$ of $\lambda \cap D$,
			$\rho$ is positively transverse to $\lambda$ and
			negatively transverse to $\lambda$ at
			the two endpoints of $s$ respectively.}
	\end{figure}
	
	At last, we consider the case that $\rho$ is an oriented properly embedded arc.
	If $\rho$ is non-essential,
	then there is an embedded disk $D \subseteq \Sigma$ with $\partial D \subseteq \rho \cup \partial \Sigma$.
	As in the previous case,
	we can still choose a regular leaf $\mu$ of $\mathcal{F}^{s}$ that has only transverse intersections with $\rho$.
	Then $\rho$ must be negatively transverse to $\mu$ at some point.
	This is also a contradiction.
\end{proof}

Because each component of $\alpha$ is positively transverse to $\mathcal{F}^{s}$,

\begin{cor}
	Each component of $\alpha$ is essential in $\Sigma$.
\end{cor}

To verify that $B(\alpha)$ contains no trivial bubble,
it suffices to show that $B(\alpha)$ has no $3$-ball complementary region.
Note that $M \setminus \setminus B(\alpha)$ is homeomorphic to $(\Sigma \setminus \setminus \alpha) \times I$.
If there is a $3$-ball complementary region of $B(\alpha)$,
then $\Sigma \setminus \setminus \alpha$ must have some disk component.
Now we exclude this case.

\begin{lm}\label{alpha}
	$\Sigma \setminus \setminus \alpha$ contains no disk component.
\end{lm}
\begin{proof}
	\begin{figure}\label{no disk}
		\centering
			\includegraphics[width=0.45\textwidth]{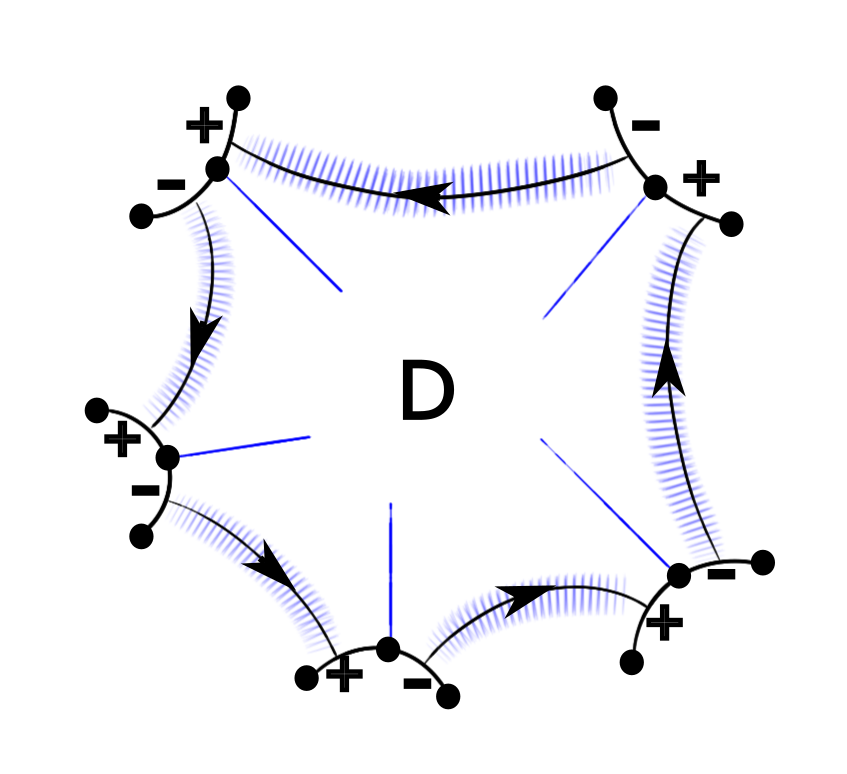}
		\caption{This picture describes the disk component $D$ of 
			$\Sigma \setminus \setminus \alpha$ assumed to exist in the proof of Lemma \ref{alpha}.
			The dots are singularities of $\mathcal{F}^{s}$ contained in $\partial \Sigma$,
			the blue lines are leaves of $\mathcal{F}^{s}$,
			and the segments labeled with $+, -$ are positive stable segments,
			negative stable segments respectively.}
	\end{figure}
	
	Assume that $\Sigma \setminus \setminus \alpha$ contains a disk component $D$.
	In the following discussions,
	the clockwise and anticlockwise orientations on $\partial D$ will always be with respect to 
	the orientation on $\Sigma$
	(then $D$ is in the left side of $\partial D$ if $\partial D$ has the anticlockwise orientation).
	Let $\gamma: I \to \Sigma$ be a component of $\alpha$ contained in $\partial D$,
	and we may assume that
	the orientation on $\gamma$ is consistent with the anticlockwise orientation on $\partial D$.
	Now draw a path $\eta: I \to \partial D$ that
	starts at $\gamma(1)$,
	goes along the anticlockwise orientation on $\partial D$,
	and ends at $\gamma(1)$.
	Recall that $\gamma(1)$ is contained in a positive stable segment in $\partial \Sigma$,
	and this positive stable segment does not contain any other endpoint of $\alpha$
	(Fact \ref{endpoint}).
	So $\eta$ first goes to a negative stable segment adjacent to the positive stable segment containing $\gamma(1)$, 
	and then goes through another arc in $\alpha$ along its orientation,
	and so on (Figure \ref{no disk}).
	So the anticlockwise orientation on $\partial D$ is consistent with 
	the orientation on each component of $\alpha$ contained in $\partial D$.
	Thus,
	the anticlockwise orientation on $\partial D$ is positively transverse to $\mathcal{F}^{s}$ 
	in $\partial D \cap \alpha$ and
	tangent to $\mathcal{F}^{s}$ in $\partial D \cap \partial \Sigma$.
	This contradicts Lemma \ref{transverse}.
	So $\Sigma \setminus \setminus \alpha$ contains no disk component.
\end{proof}

It follows that

\begin{cor}\label{no trivial bubble}
	$B(\alpha)$ contains no trivial bubble.
\end{cor}

We've verified that each component of $\alpha$ is essential and
$\Sigma \setminus \setminus \alpha$ contains no disk component.
As shown in \hyperref[S]{[S, Lemma 3.16]},
$B(\alpha)$ satisfies Conditions (b)$\sim$(d) of Definition \ref{essential}.

\begin{remark}\rm
	In \hyperref[S]{[S, Lemma 3.16]}, 
	the branched surface is assumed to have the property that the upper and lower arcs of 
	the product disks intersect efficiently in $\Sigma \times \{0\}$.
	Although our branched surface $B(\alpha)$ does not satisfy this assumption, 
	the conclusion of \hyperref[S]{[S, Lemma 3.16]} still holds for $B(\alpha)$.
	This is because the conditions verified in \hyperref[S]{[S, Lemma 3.16]} 
	depend only on the closed complement $M - \operatorname{Int}(N(B(\alpha)))$,
	and not on the efficiency of intersections.
	Hence $B(\alpha)$ satisfies Conditions (b)$\sim$(d) of Definition \ref{essential}.
\end{remark}

It remains to prove that $B(\alpha)$ contains no sink disk or half sink disk.

\begin{prop}\label{sink disk free}
	$B(\alpha)$ has no sink disk or half sink disk.
\end{prop}
\begin{proof}
	\begin{figure}\label{sink disk obstruction}
		\centering
		\subfigure[]{
			\includegraphics[width=0.32\textwidth]{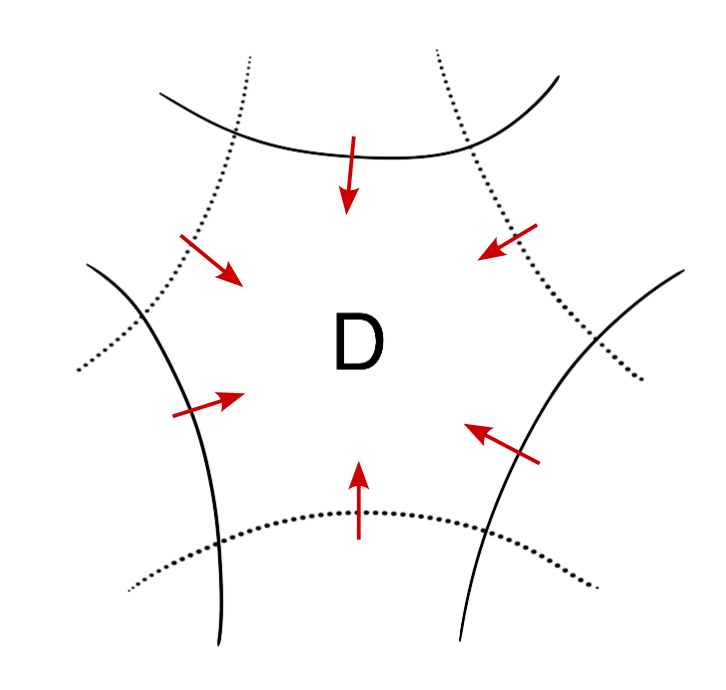}}
		\subfigure[]{
			\includegraphics[width=0.32\textwidth]{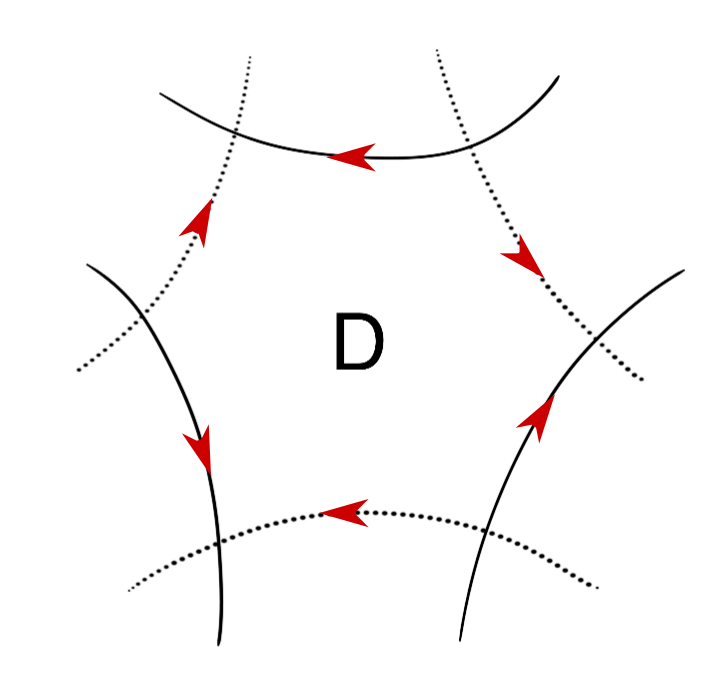}}
		\centering
		\subfigure[]{
			\includegraphics[width=0.32\textwidth]{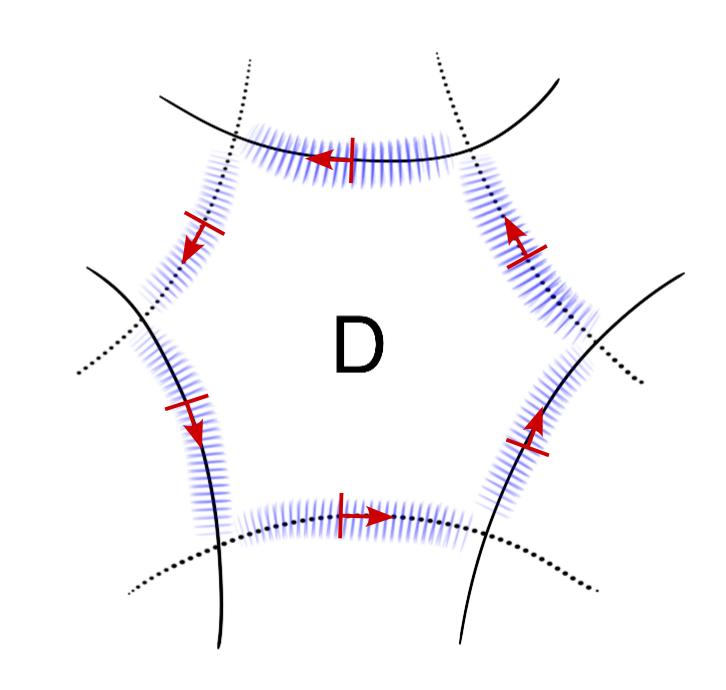}}
		\caption{The solid lines are subarcs of $\alpha \times \{0\}$ and
			the dashed lines are subarcs of $\beta \times \{0\}$.
			For a sink disk $D$,
			(a) illustrates the cusp directions at the segments of $\partial D$,
			(b) illustrates the orientations on the segments of $\partial D$,
			(c) illustrates that the anticlockwise orientation on $\partial D$ is 
			positively transverse to $\mathcal{F}^{s}$.}
	\end{figure}
	
	We first show that every product disk in $B(\alpha)$ is not a sink disk or a half sink disk.
	For every product disk $S$ of $B(\alpha)$,
	the cusp directions at both of its upper arc and lower arc point out of $S$.
	So $S$ is not a sink disk or a half sink disk.
	
	Recall from Definition \ref{B(alpha)} (c) that 
	the product disks of $B(\alpha)$ intersect the fibered surface $\Sigma \times \{0\}$ in the union of properly embedded arcs
	$(\alpha \times \{0\}) \cup (\beta \times \{0\})$,
	and that $\beta$ is isotopic to $\phi(\alpha)$ relative to the endpoints.
	Now assume that $B(\alpha)$ contains a sink disk $D$.
	Then $D \subseteq \Sigma \times \{0\}$ and
	$\partial D \subseteq (\alpha \times \{0\}) \cup (\beta \times \{0\})$.
	
	For a segment $\sigma$ of $\partial D$,
	
	$\bullet$
	Assume $\sigma \subseteq \alpha \times \{0\}$.
	Then the cusp direction at $\sigma$ points to the left side.
	Since the cusp direction at $\sigma$ points into $D$,
	$\sigma$ has the anticlockwise orientation 
	(Figure \ref{sink disk obstruction} (b), where the solid lines are subarcs of $\alpha \times \{0\}$).
	As $\sigma \subseteq \alpha \times \{0\}$,
	$\sigma$ is positively transverse to $\mathcal{F}^{s} \times \{0\} \subseteq \Sigma \times \{0\}$.
	
	$\bullet$
	Assume $\sigma \subseteq \beta \times \{0\}$.
	In this case, 
	the cusp direction at $\sigma$ points to the right side and points into $D$.
	So $\sigma$ has the clockwise orientation 
	(Figure \ref{sink disk obstruction} (b), where the dashed lines are subarcs of 
	$\beta \times \{0\}$).
	Because $\varphi$ is co-orientation-reversing and
	$\beta$ is isotopic to $\varphi(\alpha)$ relative to the endpoints,
	$\beta$ is negatively transverse to $\mathcal{F}^{s}$.
	So
	$\sigma$ is negatively transverse to $\mathcal{F}^{s} \times \{0\}$.
	
	So every segment of $\partial D$ either
	has the anticlockwise orientation and 
	is positively transverse to $\mathcal{F}^{s} \times \{0\}$,
	or has the clockwise orientation and 
	is negatively transverse to $\mathcal{F}^{s} \times \{0\}$.
	Therefore,
	the anticlockwise orientation on $\partial D$ is positively transverse to 
	$\mathcal{F}^{s} \times \{0\}$
	(Figure \ref{sink disk obstruction} (c)).
	However,
	$\partial D$ is non-essential in $\Sigma$ since it bounds the disk $D$.
	This contradicts Lemma \ref{transverse}.
	So $B(\alpha)$ contains no sink disk.
	
	Similarly,
	if $B(\alpha)$ contains a half sink disk $D$,
	then the anticlockwise orientation on $\partial D$ 
	is positively transverse to $\mathcal{F}^{s} \times \{0\}$ at every segment contained in
	$(\alpha \times \{0\}) \cup (\beta \times \{0\})$ and
	is tangent to $\mathcal{F}^{s} \times \{0\}$ at $\partial D \cap (\partial \Sigma \times \{0\})$.
	This also contradicts Lemma \ref{transverse}.
	So $B(\alpha)$ also contains no half sink disk.
	It follows that
	$B(\alpha)$ is laminar.
\end{proof}

Thus

\begin{cor}\label{laminar branched surface}
$B(\alpha)$ is a laminar branched surface.
\end{cor}

\subsection{Simple closed curves carried by boundary train tracks}\label{subsection 3.3}

Let $\tau(\alpha) = B(\alpha) \cap \partial M$.
In this subsection,
we choose some simple closed curves carried by $\tau(\alpha)$ and compute their slopes.
We first give some descriptions for $\tau(\alpha)$.

\begin{defn}\rm\label{vertical edges}
	Let $\gamma: I \to \partial \Sigma$ be a component of $\alpha$.
	Then the product disk $\gamma \times I$ intersects $\partial M$ at
	$(\{\gamma(0)\} \times I) \cup (\{\gamma(1)\} \times I)$.
	We call $\{\gamma(0)\} \times I$ (resp. $\{\gamma(1)\} \times I$)
	a \emph{positive vertical edge} (resp. \emph{negative vertical edge}) of $\tau(\alpha)$.
\end{defn}

Under the assumption of Definition \ref{vertical edges},
let $C_1, C_2$ denote the boundary components of $\Sigma$ that contain $\gamma(0), \gamma(1)$ respectively.
Recall from  Convention \ref{orientation} (b) and Definition \ref{B(alpha)},

$\bullet$
The positive orientation on $C_1$ goes from the left side of $\gamma$ to its right side at $\gamma(0)$.

$\bullet$
The positive orientation on $C_2$ goes from the right side of $\gamma$ to its left side at $\gamma(1)$.

$\bullet$
The cusp directions at the lower arc of $\gamma \times I$ point to
its left side in $\Sigma \times \{0\}$.

$\bullet$
The cusp directions at the upper arc of $\gamma \times I$ point to
its right side in $\Sigma \times \{0\}$.

See Figure \ref{type} (a) for an illustration.
Thus we have

\begin{figure}\label{type}
	\centering
	\subfigure[]{
		\includegraphics[width=0.45\textwidth]{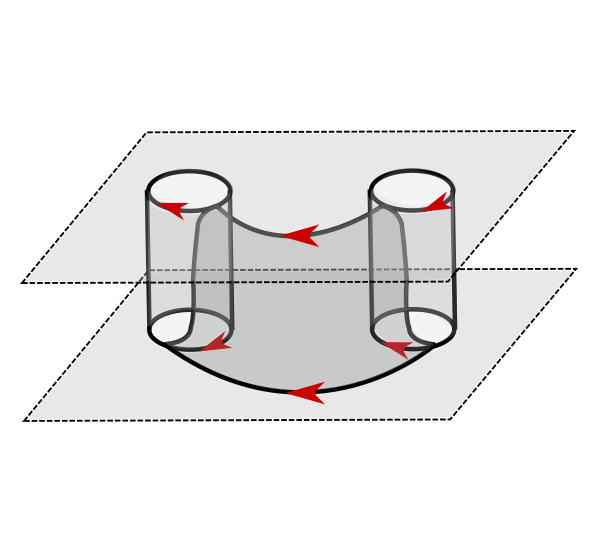}}
	\subfigure[]{
		\includegraphics[width=0.45\textwidth]{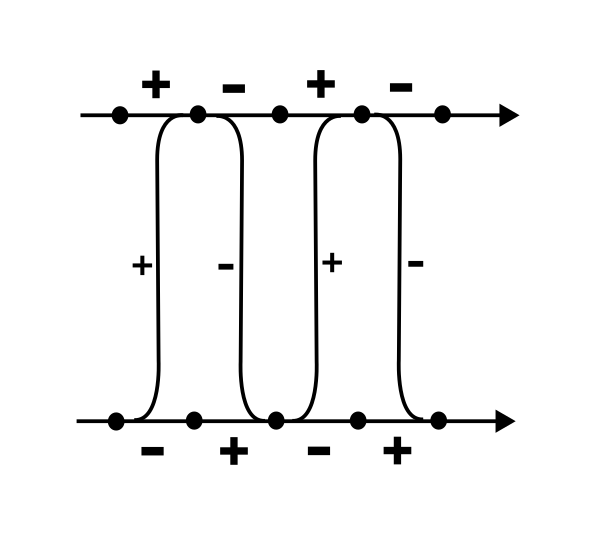}}
	\caption{(a) is a local picture of $B(\alpha)$ near a product disk $\gamma \times I$.
		It intersects $\partial M$ at two vertical segments,
		where the right one is a positive vertical edge and the left one is a negative vertical edge.
		The boundary components of $\partial \Sigma \times \{0\}$ are labeled with the positive orientations,
		and the upper and lower arcs of $\gamma \times I$ 
		are labeled with the orientations induced from $\gamma$.
		(b) is a local picture of $\tau(\alpha)$ near some vertical edges, 
		seen from the outside of $M$.
		The horizontal edges are contained in $\partial \Sigma \times \{0\}$, 
		and their positive orientations are toward the right side.
	The horizontal segments labeled with $-, +$ are negative stable segments,
	positive stable segments respectively.
    Similarly, the vertical segments labeled with $-, +$ are negative vertical edges,
positive vertical edges respectively.}
\end{figure}

\begin{fact}\rm
	(a)
	For a positive vertical edge $\{a\} \times I$ ($a \in \partial \Sigma$) of $\tau(\alpha)$,
	the cusp direction at the point $(a,0)$ is consistent with the negative orientation on
	$\partial \Sigma \times \{0\}$,
	and the cusp direction at $(a,1) = (\varphi(a), 0)$ is consistent with the positive orientation on
	$\partial \Sigma \times \{0\}$.
	For example,
	see the right one of the two vertical edges in Figure \ref{type} (a).
	
	(b)
	For a negative vertical edge $\{b\} \times I$ ($b \in \partial \Sigma$) of $\tau(\alpha)$,
	the cusp direction at $(b,0)$ is consistent with the positive orientation on
	$\partial \Sigma \times \{0\}$,
	and the cusp direction at $(b,1) = (\varphi(b), 0)$ is consistent with the negative orientation on
	$\partial \Sigma \times \{0\}$.
	For example, see the left one of the two vertical edges in Figure \ref{type} (a).
\end{fact}

For a positive or negative vertical edge $\{a\} \times I$ (where $a \in \partial \Sigma$),
we call $(a,0)$ its \emph{lower endpoint} and
call $(a,1) = (\varphi(a),0)$ its \emph{upper endpoint}.
Since $\varphi$ is co-orientation-reversing,
$\varphi$ takes all positive stable segments to negative stable segments and
takes all negative stable segments to positive stable segments.
Thus, 
for a positive vertical arc (resp. negative vertical arc),
its lower endpoint is contained in a negative stable segment (resp. positive stable segment) and
its upper endpoint is contained in a positive stable segment (resp. negative stable segment),
compare with Figure \ref{type} (b).

Let $C$ be a boundary component of $\Sigma$ and 
let $c$ denote the order of $C$ under $\varphi$ (i.e. $c = \min \{k \in \mathbb{N}_+ \mid \varphi^{k}(C) = C\}$).
Let $T$ denote the boundary component of $M$ containing $C \times \{0\}$ and let
$\tau = \tau(\alpha) \cap T$.
Let $p$ denote the number of singularities of $\mathcal{F}^{s}$ contained in $C$,
and let $v_1,\ldots,v_p$ denote the $p$ singularities of $\mathcal{F}^{s}$ contained in $C$
(consecutive along the positive orientation on $C$).
Let $q \in \mathbb{Z}$ for which $(p;q)$ is the degeneracy locus of the suspension flow of $\varphi$ on $T$.
As explained in Remark \ref{remark} (a),
$v_{j+q} = \varphi^{c}(v_j)$ (mod $p$) for each $j \in \{1,\ldots,p\}$.
We note that $2 \mid p$ since $\varphi$ is co-orientable,
and $q \equiv c \text{ } (\text{mod } 2)$ since 
$2 \mid q$ if and only if $\varphi^{c}$ is co-orientation-preserving.

\begin{defn}\rm\label{varphi triple}
	Under the assumption as above,
	we call $(c,p,q)$ the \emph{$\varphi$-triple} for $C$.
\end{defn}

Note that the triple $(c,p,q)$ depends only on the boundary component $T \subseteq \partial M$ containing $C$.

\begin{prop}\label{maximal slope}
	$\tau$ carries a simple closed curve of slope $\frac{p}{q+c}$ and
	a simple closed curve of slope $\frac{p}{q-c}$,
	where 
	$\frac{p}{q + c} = \infty$ if $q + c = 0$ and 
	$\frac{p}{q - c} = \infty$ if $q - c = 0$.
\end{prop}
\begin{proof}
Let $\mu_T, \lambda_T$ denote the meridian and longitude on $T$
(see Convention \ref{orientation} for the orientations on them).
We may assume that every $[v_{2i-1}, v_{2i}]$ is a negative stable segment and 
every $[v_{2i}, v_{2i+1}]$ is a positive stable segment.
Let $t_j = (v_j,0) \in \Sigma \times \{0\}$ for each $j \in \{1,\ldots,p\}$.

	Let $\gamma: I \to \tau$ be the path that 
	starts at $t_1$ and repeatedly performs the following steps,
	(1)
	once $\gamma$ reaches $T \cap (\Sigma \times \{0\})$,
	$\gamma$ goes along the positive orientation on $T \cap (\Sigma \times \{0\})$ until it meets 
	the lower endpoint of some positive vertical edge,
	(2)
	when $\gamma$ meets the lower endpoint of some positive vertical edge,
	$\gamma$ goes along this positive vertical edge and reaches $T \cap (\Sigma \times \{0\})$ again,
	(3)
	$\gamma$ stops by $t_1$ at its second time to meet $t_1$.
	Compare with Figure \ref{boundary train track} (a) for the picture when $\gamma$ starts.
	
	To compute the slope of $\gamma$,
	we first find which stable segment $\gamma$ reaches at its second time to intersect $C \times \{0\}$.
	Let $\gamma_0$ be the subpath of $\gamma$ that starts at $t_1$ and stops at the second time it meets $C \times \{0\}$.
	We refer to each endpoint of every stable segment as a \emph{dot}
	(the dots are represented as dotted points in Figure \ref{boundary train track}).
	Since $c = |T \cap (\Sigma \times \{0\})|$, 
	the subpath $\gamma_0$ undergoes the steps (1), (2) exactly $c$ times,
	passing through $c-1$ dots in total. 
	Since $\varphi^{c}(v_1) = v_{q+1}$,
	it follows that the dot $(v_1,1) \in T$ is identified with 
	$(v_{q+1},0) = t_{q+1}$ in $M$.
	Therefore, if we start at $t_{q+1} = (v_1,1)$, move along the positive orientation on $C \times \{0\}$,
	and pass through $c-1$ dots, we arrive at the stable segment where $\gamma_0$ terminates. 
	Thus, the subpath $\gamma_0$ stops at some point in the positive stable segment $[t_{q+c}, t_{q+c+1}]$ (mod $p$).  
	
	If $t_{q+c+1} = t_1$,
	then after reaching the endpoint of $\gamma_0$,
	the path $\gamma$ goes along the positive orientation on $[t_p,t_1]$ and stops at $t_1$.
	Otherwise,
	$\gamma$ continues repeating the steps (1) and (2) as described above,
	passing through $t_{2(q+c)+1}, t_{3(q+c)+1}, \ldots$,
	until it arrives at $t_1$ at last.
	Because \[\min\{k \in \mathbb{Z}_{\geqslant 1} \mid k(q+c) +1 \equiv 1 \text{ }(\text{mod } p)\} = \frac{p}{\gcd (p,q+c)},\]
	by convention $\gcd(p,0) = p$,
	the path $\gamma$ passes $t_{k(q+c)+1}$ for all $2 \leqslant k \leqslant \frac{p}{\gcd (p,q+c)} - 1$
	and arrives at $t_1$ 
	the $(\frac{p}{\gcd (p,q+c)})^{\text{th}}$ time it meets $C \times \{0\}$.
	It follows that
	$$\langle \gamma, \lambda_T \rangle = \frac{p}{\gcd (p,q+c)}.$$

	The meridian on $T$ is represented by the path that starts at $t_1$,  
	travels along $\{v_1\} \times I$ to $(v_1,1) = (v_{q+1},0) = t_{q+1}$,  
	and then moves from $t_{q+1}$ to $t_1$ along the negative (resp. positive) orientation on $C \times \{0\}$ 
	when $q > 0$ (resp. $q < 0$).  
	Thus, as $\gamma_0$ moves from $t_1$ to \( t_{q+c+1} \),  
	it can be viewed as a meridian path with a \( \frac{q+c}{p} \) twist along  
	the positive orientation on the longitude
	(if \( q+c < 0 \), this corresponds to a \( -\frac{q+c}{p} \) twist along the negative orientation).  
	Similarly, for each subpath of \( \gamma \) from \( t_{k(q+c)+1} \) to \( t_{(k+1)(q+c)+1} \)  
	(where \( 1 \leqslant k \leqslant \frac{p}{\gcd (p,q+c)} - 1 \)),  
	there is a \( \frac{q+c}{p} \) twist along the positive orientation on the longitude.
	Hence
	$$\langle \mu_T, \gamma \rangle = 
	\frac{q+c}{p} \cdot \langle \gamma, \lambda_T \rangle =
	\frac{q+c}{\gcd (p,q+c)}.$$
	Therefore,
	$$\text{slope}(\gamma) = \frac{\langle \gamma, \lambda_T \rangle}{\langle \mu_T, \gamma \rangle} = 
	\frac{(\frac{p}{\gcd (p,q+c)})}{(\frac{q+c}{\gcd (p,q+c)})} = \frac{p}{q+c}.$$
	
	\begin{figure}\label{boundary train track}
		\centering
		\subfigure[]{
			\includegraphics[width=0.32\textwidth]{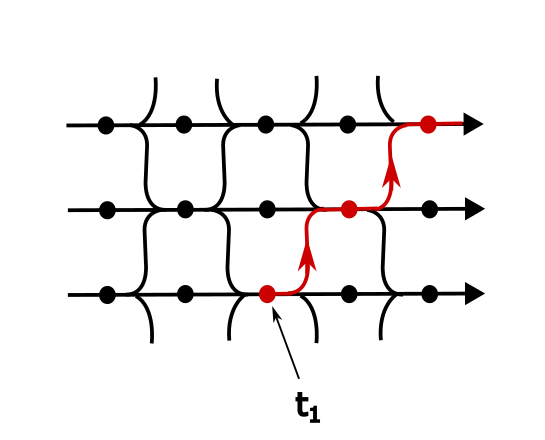}}
		\subfigure[]{
			\includegraphics[width=0.32\textwidth]{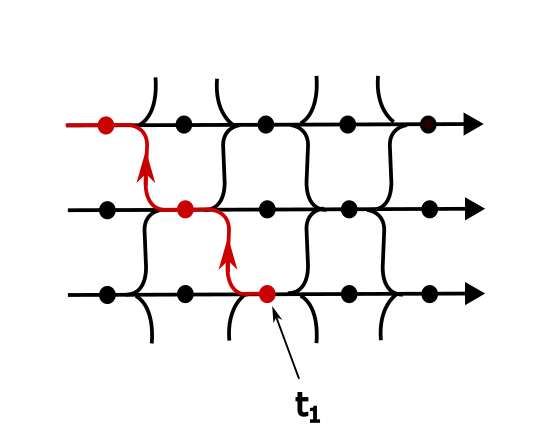}}
		\caption{The horizontal edges are contained in $\partial \Sigma \times \{0\}$,
			where the positive orientations on them are toward the right.
			In the proof of Proposition \ref{maximal slope},
			we construct two paths $\gamma, \nu: I \to \tau$ with $\text{slope}(\gamma) = \frac{p}{q+c}$,
			$\text{slope}(\nu) = \frac{p}{q-c}$.
			(a) describes $\gamma$ when it starts,
			and (b) describes $\nu$ when it starts.}
	\end{figure}
	
	Next,
	we choose a simple closed curve $\nu$ carried by $\tau$ with $\text{slope}(\nu) = \frac{p}{q-c}$.
	Let $\nu: I \to \tau$ be the path starting at $t_1$ such that
	(1)
	when $\nu$ gets to $T \cap (\Sigma \times \{0\})$,
	$\nu$ goes along the negative orientation on $T \cap (\Sigma \times \{0\})$ until it meets 
	the lower endpoint of some negative vertical edge,
	(2)
	when $\nu$ meets the lower endpoint of some negative vertical edge,
	$\nu$ goes along this negative vertical edge and gets to $T \cap (\Sigma \times \{0\})$ again,
	(3)
	$\nu$ stops by $t_1$ at the second time to meet $t_1$.
	Compare with Figure \ref{boundary train track} (b) for the picture when $\nu$ starts.
	
	Next, we describe where $\nu$ reaches $C \times \{0\}$ for the second time.
	Let $\nu_0$ be the subpath of $\nu$ that starts at $t_1$ and stops at the second time to meet $C \times \{0\}$.
	As in the case of $\gamma_0$ discussed above,  
	$\nu_0$ undergoes the steps (1), (2) exactly $c$ times and passes through $c-1$ dots in total.
	Thus,
	the stable segment where \(\nu_0\) terminates can be reached by
	going from $t_{q+1}$, moving along the negative orientation on $C \times \{0\}$ and passing through $c-1$ dots.
	It follows that \(\nu_0\) stops in the positive stable segment $[t_{q-c+1}, t_{q-c+2}]$ (mod $p$).
	
	After reaching the positive stable segment $[t_{q-c+1}, t_{q-c+2}]$ (mod $p$)
	at its second time to meet $C \times \{0\}$,
	the path $\nu$ continues along the negative orientation on $C \times \{0\}$ toward $t_{q-c+1}$.
	If $t_{q-c+1} = t_1$,
	then $\nu$ stops at this time.
	Otherwise,
	$\nu$ repeats the above steps,
	passing through $t_{2(q-c)+1}, t_{3(q-c)+1}, \ldots$,
	and arrives at $t_1$ at last.
	
	Similar to the case of $\gamma$,
	the path $\nu$ passes $t_{k(q-c)+1}$ for all $2 \leqslant k \leqslant \frac{p}{\gcd (p,q-c)} - 1$ (by convention, $\gcd(p,0) = p$)
	and arrives at $t_1$ the $(\frac{p}{\gcd (p,q-c)})^{\text{th}}$ time it meets $C \times \{0\}$.
	Thus
    	$$\langle \nu, \lambda_T \rangle = \frac{p}{\gcd (p,q-c)}.$$
    As $\nu_0$ moves from $t_{k(q-c)+1}$ to \( t_{(k+1)(q-c)+1} \) ($2 \leqslant k \leqslant \frac{p}{\gcd (p,q-c)} - 1$),  
    it can be regarded as a meridian path with a \( \frac{q-c}{p} \) twist along  
    the positive orientation of the longitude,
    and therefore
    	$$\langle \mu_T, \nu \rangle = 
    	\frac{q-c}{p} \cdot \langle \nu, \lambda_T \rangle =
    	\frac{q-c}{\gcd (p,q-c)}.$$
    It follows that
    	$$\text{slope}(\nu) = \frac{\langle \nu, \lambda_T \rangle}{\langle \mu_T, \nu\rangle} = 
    	\frac{(\frac{p}{\gcd (p,q-c)})}{(\frac{q-c}{\gcd (p,q-c)})} = \frac{p}{q-c}.$$
\end{proof}

\subsection{Slopes realized by the boundary train tracks}\label{subsection 3.4}

We first briefly review some ingredients for measures on train tracks.
Let $\tau$ be a train track on a compact orientable surface $S$.
We denote by $E(\tau)$ the set of edges of $\tau$.

$\bullet$
A \emph{measure} $m: E(\tau) \to \mathbb{R}_{\geqslant 0}$ on $\tau$ is
an assignment of nonnegative numbers to $E(\tau)$ that satisfies the cusp relation
(Figure \ref{cusp relation} (a)) at each cusp,
i.e. for any three edges $x,y,z \in E(\tau)$ that has a common endpoint $p$,
if the cusp direction at $p$ points toward $z$,
then $m(x) + m(y) = m(z)$.

\begin{figure}\label{cusp relation}
	\centering
	\subfigure[]{
		\includegraphics[width=0.32\textwidth]{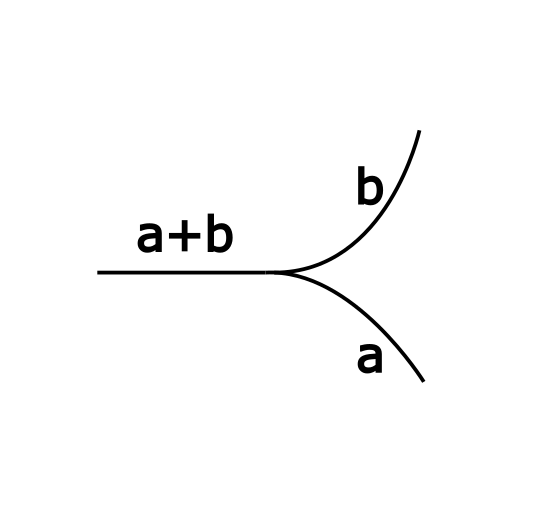}}
	\subfigure[]{
		\includegraphics[width=0.32\textwidth]{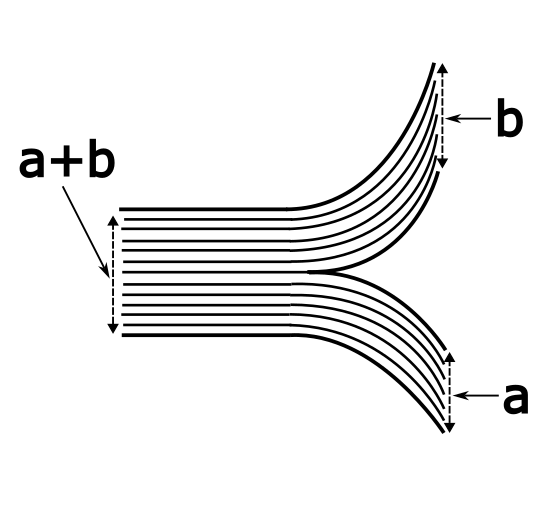}}
	\caption{(a) The cusp relation of a measure at a cusp point.
	(b) The companion lamination of the measure as given in (a).}
\end{figure}

$\bullet$
Measures on $\tau$ are in one-to-one correspondence with
measured laminations carried by $\tau$.
For a measure $m$ on $\tau$,
call the corresponding measured lamination the \emph{companion lamination} of $m$
(compare with Figure \ref{cusp relation} (b)).

$\bullet$
$\tau$ is \emph{orientable} if its edges have continuously varying orientations.
Assume that $\tau$ is orientable,
and let $\Lambda$ be a measured lamination carried by $\tau$.
Then $\Lambda$ is orientable.
Fix an orientation on $\tau$,
then $\Lambda$ has an \emph{induced orientation} from $\tau$ such that 
each leaf of $\Lambda$ is oriented consistently with the orientation on $\tau$.
Moreover,
any closed oriented curve (on the surface $S$) has
a well-defined algebraic intersection number with $\Lambda$.

We refer the reader to \hyperref[FM]{[FM, Chapter 15]} for more details.

\begin{notation}\rm
	Let $\tau$ be a train track on a surface $S$.
	
	(a)
	Let $m_1, m_2$ be two measures on $\tau$.
	Then
	$m_1 + m_2$ denotes the measure on $\tau$ with
	$(m_1 + m_2)(e) = m_1(e) + m_2(e)$ for each $e \in E(\tau)$.
	
	(b)
	Let $\rho$ be a simple closed curve carried by $\tau$ and
	let $t \in \mathbb{R}_+$.
	Then we can regard $\rho \times [0,t]$ as a measured lamination carried by $\tau$.
	Let $\rho(t)$ denote the measure on $\tau$ for which
	$\rho \times [0,t]$ is the companion measured lamination of $\rho(t)$.
\end{notation}

\begin{prop}\label{realize}
	Let $C$ be a component of $\partial \Sigma$ and
	let $T$ be a boundary component of $M$ containing $C \times \{0\}$.
	Let $\tau = \tau(\alpha) \cap T$.
	Let $(c,p,q)$ denote the $\varphi$-triple for $C$ (Definition \ref{varphi triple}).
	
	(a)
	If $q > c > 0$,
	then $\tau$ realizes all rational slopes in 
	$(-\infty,\frac{p}{q+c}) \cup (\frac{p}{q-c}, +\infty) \cup \{\infty\}$.
	
	(b)
	If $q = c > 0$,
	then $\tau$ realizes all rational slopes in 
	$(-\infty,\frac{p}{2q})$.
	
	(c)
	If $c > q \geqslant 0$,
	then $\tau$ realizes all rational slopes in $(-\frac{p}{c-q},\frac{p}{q+c})$.
	
	(d)
	If $-c < q < 0$,
	then $\tau$ realizes all rational slopes in $(-\frac{p}{|q|+c},\frac{p}{c-|q|})$.
	
	(e)
	If $q = -c < 0$,
	then $\tau$ realizes all rational slopes in $(-\frac{p}{2|q|},+\infty)$.
	
	(f)
	If $q < -c < 0$,
	then $\tau$ realizes all rational slopes in $(-\infty, -\frac{p}{|q|-c}) \cup (-\frac{p}{|q|+c},+\infty) \cup \{\infty\}$.
\end{prop}
\begin{proof}
	We only prove (a).
	The proofs of (b)$\sim$(f) are similar to (a).
	Let $\mu_T, \lambda_T$ denote the meridian and longitude on $T$.
	For every vertical edge $\{a\} \times I$ in $\tau$,
	its \emph{upward orientation} (resp. \emph{downward orientation}) refers to
	the orientation on it consistent with the increasing orientation (resp. decreasing orientation)
	on the second coordinate.
	
	Assume $q > c > 0$.
	To show that \(\tau\) realizes all rational slopes in  
	\(\left(-\infty,\frac{p}{q+c}\right) \cup \left(\frac{p}{q-c}, +\infty\right) \cup \{\infty\}\),  
	we divide the proof into two parts.  
	In the first part, we show that $\tau$ realizes all rational slopes in 
	$[0, \tfrac{p}{q+c})$,  
	and in the second part, we show that $\tau$ realizes all rational slopes in 
	$\left(-\infty, 0\right) \cup \{\infty\} \cup \left(\tfrac{p}{q-c}, +\infty\right)$.
	
	\begin{part 1*}\rm
	We first orient $\tau$ so that
	$T \cap (\Sigma \times \{0\})$ has the positive orientation,
	every positive vertical edge in $\tau$ has the upward orientation,
	and every negative vertical edge in $\tau$ has the downward orientation.
	
	By Proposition \ref{maximal slope} (a),
	$\tau$ carries a simple closed curve $\gamma$ of slope $\frac{p}{q+c}$.
	Let $u = \frac{p}{\gcd(p, q+c)}$,
	$v = \frac{q+c}{\gcd(p, q+c)}$.
	Then $\frac{p}{q+c} = \frac{u}{v}$, $u, v > 0$, and $\gcd(u,v) = 1$.
	We orient $\gamma$ with the orientation induced from $\tau$,
	then
	$\langle \gamma, \lambda_T \rangle = u$,
	$\langle \mu_T, \gamma \rangle = v$.
	
	For every edge $e$ of $\tau$, 
	we can choose 
	a simple closed curve $\rho_e$ carried by $\tau$ such that $\rho_e$ contains $e$ and
	$\text{slope}(\rho_e) = 0$.
	We orient each $\rho_e$ with the orientation induced from $\tau$.
	Then $\langle \rho_e, \lambda_T \rangle = 0$,
	$\langle \mu_T, \rho_e \rangle = 1$.
		
	Let $m_0$ denote the measure $\sum_{e \in E(\tau)} \rho_e(1)$,
	and let $\Lambda_0$ denote the companion lamination of $m_0$
	(with the orientation induced from $\tau$).
	Then
	$\langle \Lambda_0, \lambda_T\rangle = 0$,
	and thus $\tau$ realizes the slope $0$.

It remains to show that $\tau$ realizes all rational slopes in $(0,\tfrac{p}{q+c})$.
	We define a family of one-parameter measures $m_1(t)$ with $t \in (0,+\infty)$ such that
	$$m_1(t) = \gamma(1) + \sum_{e \in E(\tau)} \rho_e(t).$$
	Let $\Lambda_1(t)$ denote the companion lamination of $m_1(t)$,
	and we assign $\Lambda_1(t)$ the orientation induced from $\tau$.
	Let $N = |E(\tau)|$.
	Then
	$$\frac{\langle \Lambda_1(t), \lambda_T\rangle}{\langle \mu_T, \Lambda_1(t) \rangle} =
	\frac{u}{v+tN}.$$
	For any rational number $x \in (0,\frac{p}{q+c})$,
	we can choose some $t > 0$ so that
	$$\frac{u}{v+tN} = x,$$
	which implies that $\tau$ realizes the slope $x$.
\end{part 1*}
	
	\begin{part 2*}\rm
	Now we prove that $\tau$ realizes all rational slopes in $(-\infty, 0) \cup \{\infty\} \cup (\frac{p}{q-c}, +\infty)$.
	We re-orient $\tau$ so that
	$T \cap (\Sigma \times \{0\})$ has the negative orientation,
	every positive vertical edge in $\tau$ has the downward orientation,
	and every negative vertical edge in $\tau$ has the upward orientation.
	
	By Proposition \ref{maximal slope} (a),
	$\tau$ carries a simple closed curve $\nu$ of slope $\frac{p}{q-c}$.
	We orient $\nu$ with the orientation induced from $\tau$.
	Let $r= \frac{p}{\gcd(p, q-c)}$,
	$s = \frac{q-c}{\gcd(p, q-c)}$.
	Then $\frac{p}{q-c} = \frac{r}{s}$, $r, s > 0$, and $\gcd(r,s) = 1$.
	Moreover, we have
	$\langle \nu, \lambda_T \rangle = r$,
	$\langle \mu_T, \nu \rangle = s$.
	
	For every edge $e \in E(\tau)$, 
	we can choose a simple closed curve $\eta_e$ carried by $\tau$ that contains $e$ and
	has slope $0$.
	We orient each $\eta_e$ with the orientation induced from $\tau$.
	Because $T \cap (\Sigma \times \{0\})$ has the negative orientation,
	$\langle \eta_e, \lambda_T \rangle = 0$,
	$\langle \mu_T, \eta_e \rangle = -1$.
	
	We define a family of one-parameter measures $m_2(t)$ with $t \in (0,+\infty)$ such that
	$$m_2(t) = \nu(1) + \sum_{e \in E(\tau)} \eta_e(t).$$
	Let $\Lambda_2(t)$ denote the companion lamination of $m_2(t)$,
	with the orientation induced from $\tau$.
	Recall that $|E(\tau)| = N$, so
	$$\frac{\langle \Lambda_2(t), \lambda_T\rangle}{\langle \mu_T, \Lambda_2(t) \rangle} =
	\frac{r}{s-tN}.$$
	If we choose $t = \frac{s}{N}$,
	then $\frac{r}{s-tN} = \infty$.
	For any rational number $x \in (\frac{p}{q-c}, +\infty)$,
	we can choose some $0 < t < \frac{s}{N}$ so that 
	$\frac{r}{s-tN} = x$.
	Similarly, for any rational number $x \in (-\infty, 0)$,
	we can choose some $t > \frac{s}{N}$ so that 
	$\frac{r}{s-tN} = x$.
	Thus,
	$\tau$ realizes all rational slopes in 
	$(-\infty, 0) \cup \{\infty\} \cup (\frac{p}{q-c}, +\infty)$.
	\end{part 2*}

Combining the above two parts,
it follows that all rational slopes in \(\left(-\infty,\frac{p}{q+c}\right) \cup \left(\frac{p}{q-c}, +\infty\right) \cup \{\infty\}\) 
are realized by $\tau$.
\end{proof}

\subsection{The proof of Theorem \ref{main}}\label{proof}

We choose a fibered neighborhood $N(B)$ of $B$.
As in Subsection \ref{subsection 2.2}, the collapsing map from $N(B)$ to $B$ is denoted by $\pi: N(B) \to B$.
In the following discussions, 
we may assume that all laminations carried by $B$ are transverse to the interval fibers of $N(B)$.

We first explain that 
$B(\alpha)$ carries no torus.
Assume that $B(\alpha)$ carries a torus $T$.
Since all product disks in $B(\alpha)$ intersect $\partial M$ and $T$ is a closed surface,
$\pi(T)$ contains no product disk,
and thus $\pi(T) \subseteq \Sigma \times \{0\}$.
However,
$\Sigma \times \{0\}$ carries no closed surface,
which contradicts $\pi(T) \subseteq \Sigma \times \{0\}$.
So $B(\alpha)$ carries no torus.

Choose a multislope $\textbf{s} = (s_1, \ldots, s_k) \in (\mathbb{Q} \cap \{\infty\})^{k}$ contained 
in the multi-inverval as given in Theorem \ref{main}.
Combining Corollary \ref{laminar branched surface}, Proposition \ref{realize} with Theorem \ref{boundary},
since $B(\alpha)$ carries no torus,
$B(\alpha)$ fully carries an essential lamination $\mathcal{L}_{\textbf{s}}$ in $M$ such that
$\mathcal{L}_{\textbf{s}}$ intersects each $T_i$ in a collection of simple closed curves of slope $s_i$.
As $B(\alpha)$ has product complementary regions,
each complementary region of $\mathcal{L}_{\textbf{s}}$ is an $I$-bundle over 
a surface.
We can extend $\mathcal{L}_{\textbf{s}}$ to a foliation $\mathcal{F}_{\textbf{s}}$ in $M$ that
intersects each $T_i$ in a foliation by simple closed curves of slope $s_i$.
We can then extend $\mathcal{F}_{\textbf{s}}$ to 
a foliation $\widehat{\mathcal{F}_{\textbf{s}}}$ in $M(\textbf{s})$.

Next, we explain that \(\widehat{\mathcal{F}_{\textbf{s}}}\) is a taut foliation in \(M(\textbf{s})\).  
The union of the core curves of the filling solid tori is transverse to \(\widehat{\mathcal{F}_{\textbf{s}}}\);
we now verify that it intersects every leaf of \(\widehat{\mathcal{F}_{\textbf{s}}}\).  
Let \(\lambda\) be a leaf of \(\mathcal{L}_{\textbf{s}}\).  
If \(\pi(\lambda)\) contains a product disk of \(B\), 
then \(\pi(\lambda)\) must intersect \(\partial M\), since every product disk intersects $\partial M$.
Otherwise, if \(\pi(\lambda)\) contains no product disk, 
then \(\lambda\) is carried by \(\Sigma \times \{0\}\), 
which still implies that \(\pi(\lambda)\) intersects \(\partial M\).  
In either case, \(\pi(\lambda)\) has nonempty intersection with \(\partial M\), and thus \(\lambda \cap \partial M \neq \emptyset\).  
It follows that every leaf of \(\mathcal{L}_{\textbf{s}}\) intersects \(\partial M\).  
Since each complementary region of \(\mathcal{L}_{\textbf{s}}\) is an \(I\)-bundle over a surface,
and its two boundary leaves are contained in \(\mathcal{L}_{\textbf{s}}\), 
it follows that every leaf of \(\mathcal{F}_{\textbf{s}} - \mathcal{L}_{\textbf{s}}\) also intersects \(\partial M\).  
Thus, all leaves of $\mathcal{F}_{\textbf{s}}$ intersect $\partial M$,
and so the union of the core curves of the filling solid tori intersects all leaves of $\widehat{\mathcal{F}_{\textbf{s}}}$. 
Therefore, \(\widehat{\mathcal{F}_{\textbf{s}}}\) is a taut foliation.  

Recall that $B(\alpha)$ is co-orientable.
A co-orientation on $B(\alpha)$ induces a co-orientation on
$\mathcal{L}_{\textbf{s}}$ and on $\mathcal{F}_{\textbf{s}}$.
It follows that
$\widehat{\mathcal{F}_{\textbf{s}}}$ is co-orientable.
This completes the proof of Theorem \ref{main}.

\section{The proofs of Proposition \ref{pretzel} and Proposition \ref{no interior singularity}}\label{section 4}

We begin by proving Proposition \ref{pretzel}.

Let $q \in \mathbb{N}$ with $q \geqslant 3$,
and let $K$ be the $(-2,3,2q+1)$-pretzel knot in $S^{3}$.
Let $g(K)$ denote the Seifert genus of $K$ (then $g(K) = q + 2$).
Let $X = S^{3} - Int(N(K))$,
let $S$ be a fibered surface of $X$,
and let $\phi: S \to S$ denote the pseudo-Anosov monodromy of $X$.
Fix an orientation on $S^{3}$ so that $\phi$ is right-veering.

\begin{pretzel}\label{pretzel monodromy}
(a)
$\phi$ is co-orientable and co-orientation-reversing.

(b)
$K$ has degeneracy slope $4g(K)-2$.

(c)
All rational slopes in $(-\infty, 2g(K) - 1)$ are CTF surgery slopes of $K$.

(d)
For each $n \geqslant 2$,
the $n$-fold cyclic branched cover of $K$ admits a co-orientable taut foliation.
\end{pretzel}

\begin{proof}[The proof of (a)]
Let $\mu, \delta$ denote the meridian and degeneracy slope on $K$ respectively.
Since $\phi$ is right-veering,
we have $\mu \ne \delta$.
Let $\widetilde{X_2}$ denote the double cyclic cover of $X$,
and let $\Sigma_2(K)$ denote the double branched cover of $K$.
Then $\widetilde{X_2}$ is the mapping torus of $S$ with monodromy $\phi^{2}$.

Let $T = \partial X$ and $\widetilde{T} = \partial \widetilde{X_2}$,
and let $\widetilde{\mu}$ denote the meridian on $\widetilde{T}$.
Let $\Phi$ be the suspension flow of $\phi$ in $X$,
and let $\widetilde{\Phi}$ be the suspension flow of $\phi^{2}$ in $\widetilde{X_2}$.
As in Section \ref{section 1},
the degeneracy loci of $\Phi, \widetilde{\Phi}$ on $T, \widetilde{T}$ are denoted by $d(T)$ and $d(\widetilde{T})$, respectively.
Because $\widetilde{X_2}$ is the double cyclic cover of $X$,
we have $\Delta(\widetilde{\mu},d(\widetilde{T})) = 2\Delta(\mu,d(T))$.
As $K$ is a hyperbolic L-space knot in $S^{3}$,
we have $\Delta(\mu, \delta) = 1$,
and the degeneracy locus $d(T)$ has multiplicity $1$ (see Case \ref{case2}).
It follows that
\[\Delta(\widetilde{\mu},d(\widetilde{T})) = 2 \Delta(\mu,d(T)) = 2 \Delta(\mu, \delta) = 2.\]

As described in \hyperref[Fr]{[Fr, Section 1]},
the suspension flow $\Phi$ induces a flow $\Upsilon$ in $\Sigma_2(K)$ such that
$\Upsilon \mid_{\Sigma_2(K) - K}$ can be identified with $\widetilde{\Phi} \mid_{\operatorname{Int}(\widetilde{X_2})}$.
In addition,
the weak stable and unstable foliations of $\widetilde{\Phi}$ in $\widetilde{X_2}$ induce 
a pair of singular foliations $\mathcal{E}^{s}, \mathcal{E}^{u}$ in $\Sigma_2(K)$.
Let $\lambda^{s}, \lambda^{u}$ denote the leaves of $\mathcal{E}^{s}, \mathcal{E}^{u}$ containing $K$, respectively.
Since $\Delta(\widetilde{\mu},d(\widetilde{T})) = 2$,
both of $\lambda^{s}, \lambda^{u}$ have exactly two prongs at $K$;
that is,
each of $\lambda^{s} - K, \lambda^{u} - K$ consists of exactly two components.
As illustrated in \hyperref[Fr]{[Fr, Section 1]},
$\Upsilon$ is a pseudo-Anosov flow.
Since $\mathcal{E}^{s}, \mathcal{E}^{u}$ are induced from 
the weak stable and unstable foliations of $\widetilde{\Phi}$, respectively,
they are exactly the weak stable and unstable foliations of $\Upsilon$, respectively.

It is explained in \hyperref[BoS]{[BoS, A.3, A.4]} (and summarized in \hyperref[BoS]{[BoS, page 258, Theorem A.6]}) that
the manifold $\Sigma_2(K)$ is a Seifert fibered $3$-manifold with base orbifold \(S^2(2,3,2q+1)\),
i.e. $S^{2}$ with three cone points of index $2, 3, 2q+1$ respectively.
Since $\Sigma_2(K)$ is a Seifert fibered $3$-manifold,
it follows from \hyperref[BF]{[BF, Theorem 4.1]} and \hyperref[Ba]{[Ba]} that 
$\Upsilon$ is an Anosov flow.
In fact,
\hyperref[BF]{[BF, Theorem 4.1]} indicates that, up to a finite cover,
every pseudo-Anosov flow in a Seifert fibered $3$-manifold is 
topologically equivalent to the geodesic flow on the unit tangent bundle of some closed hyperbolic surface,
which is necessarily an Anosov flow.
It follows that $\mathcal{E}^{s}, \mathcal{E}^{u}$ are covered by 
some regular $2$-dimensional foliations in a finite cover of $\Sigma_2(K)$,
and therefore, $\mathcal{E}^{s}, \mathcal{E}^{u}$ themselves are regular $2$-dimensional foliations,
which implies that $\Upsilon$ must be an Anosov flow.

We note that every leaf of the weak stable foliation $\mathcal{E}^{s}$ is homeomorphic to $\mathbb{R}^{2}$ or $\mathbb{R} \times S^{1}$ (see, for example, \hyperref[Ba]{[Ba, Theorem 2.1]}),
which implies that $\mathcal{E}^{s}$ contains no compact leaf.
As illustrated in \hyperref[Br]{[Br, Corollary 7]},
$\mathcal{E}^{s}$ can be $C^{0}$-isotoped to a taut foliation $\mathcal{E}$ transverse to the Seifert fibers of $\Sigma_2(K)$.
Because the base orbifold of $\Sigma_2(K)$ is orientable,
the Seifert fibers of $\Sigma_2(K)$ have continuously varying orientations, 
and these orientations define a co-orientation on $\mathcal{E}$.
The co-orientation on $\mathcal{E}$ induces continuously varying leafwise orientations on $\mathcal{E}$,
which determine continuously varying leafwise orientations on $\mathcal{E}^{s}$,
along with an induced co-orientation on $\mathcal{E}^{s}$.
Since $\mathcal{E}^{s}$ is co-orientable,
it follows that $\phi^{2}$ is co-orientable and co-orientation-preserving,
and therefore $\phi$ is co-orientable.
Since $\phi$ is right-veering and $K$ is a knot in $S^{3}$,
$\phi$ can only be co-orientation-reversing.
\end{proof}

\begin{proof}[The proof of (b)]	
The construction of Fintushel and Stern in \hyperref[FS]{[FS]} implicitly shows that 
the $(-2,3,7)$-pretzel knot has degeneracy slope $18$.
Their construction can be generalized to the all $(-2,3,2k+1)$-pretzel knots with $k \in \mathbb{Z}_{\geqslant 3}$ and 
implies that their degeneracy slope is $4g-2$ (where $g$ is the genus).
We give another proof here using (a).

Let $\mathcal{F}^{s}$ denote the stable foliation of $\phi$.
Since $\Upsilon$ is an Anosov flow,
$\Upsilon$ has no singular orbit.
Thus, $\mathcal{F}^{s}$ has no singularity in $\operatorname{Int}(S)$,
and therefore each singularity of $\mathcal{F}^{s}$ has three separatrices.
It follows from the Euler-Poincar\'e formula (see, for example, \hyperref[FM]{[FM, Proposition 11.4]}) that,
the stable foliation of $\phi$, denoted $\mathcal{F}^{s}$, has $- 2 \chi(S) = 4g(K) - 2$ singularities in $\partial S$.
Here, 
we note that \hyperref[FM]{[FM, Subsection 11.2.2]} focuses only on closed surfaces,
so we provide additional explanation for why the formula holds for $(S, \mathcal{F}^{s})$.
Let $\widehat{S}$ be the closed surface obtained by collapsing $\partial S$ to a single point,
and let $\mathcal{H}$ be the singular foliation on $\widehat{S}$ induced by $\mathcal{F}^{s}$.
Then, \hyperref[FM]{[FM, Subsection 11.2.2]} applies to $(\widehat{S}, \mathcal{H})$ and implies that
the unique singularity of $\mathcal{H}$ has $2 \chi(\widehat{S}) - 2 = 4g(K) - 2$ prongs.
Hence $\mathcal{F}^{s}$ has $4g(K) - 2$ singularities in $\partial S$.

Note that $\delta > 0$ since $\phi$ is right-veering.
As $\Delta(\mu, \delta) = 1$,
we have
$\delta = 4g(K) - 2$.
\end{proof}

\begin{proof}[The proof of (c)]
	This follows from (b) and Corollary \ref{open book} directly.
\end{proof}

\begin{proof}[The proof of (d)]
	Let $\widetilde{X_n}$ denote the $n$-fold cyclic cover of $X$,
	which is also the mapping torus of $S$ with monodromy $\phi^{n}$.
	If $n$ is even,
	then $\phi^{n}$ is co-orientation-preserving,
	so the result can be deduced from Theorem \ref{co-orientation-preserving} directly.
	Now we assume that $n$ is odd
	(then $\phi^{n}$ is co-orientation-reversing).
	We denote $4g(K) - 2$ by $p$.
	There is a unique $a \in \mathbb{Z}_{\geqslant 0}$ with
	$-\frac{1}{2}p < n - ap \leqslant \frac{1}{2}p$,
	and let $b = n - ap$.
	We fix the canonical coordinate system on $\partial \widetilde{X_n}$.
	By Remark \ref{remark} (a),
	$(p; b)$ is
	the degeneracy locus of the suspension flow of $\phi^{n}$ on $\partial \widetilde{X_n}$.
	Because $\phi$ is right-veering, we have $\frac{p}{b} > 0$ and thus $b \geqslant 1$.
	
	Applying Corollary \ref{open book} to $\widetilde{X_n}$,
	a slope $s$ on $\partial \widetilde{X_n}$ is a CTF filling slope if
	$s \notin [\frac{p}{b+1}, \frac{p}{b-1}]$ (when $b > 1$) or
	$s \notin [\frac{p}{2}, + \infty) \cup \{\infty\}$ (when $b = 1$).
	The inverse image of the slope $\infty$ on $\partial X$ is
	the slope $- \frac{1}{a}$ on $\partial \widetilde{X_n}$.
	When $b = 1$,
	we have $a \geqslant 1$ since $n > 1$,
	and thus $- \frac{1}{a} \notin [\frac{p}{2}, + \infty) \cup \{\infty\}$.
	When $b > 1$,
	we have
	$\frac{p}{b+1}, \frac{p}{b-1} \in (0, +\infty)$ and $- \frac{1}{a} \in (-\infty, 0) \cup \{\infty\}$,
	which implies
	$- \frac{1}{a} \notin [\frac{p}{b+1}, \frac{p}{b-1}]$.
	The result follows directly.
\end{proof}

Proposition \ref{no interior singularity} basically follows from the proof of Proposition \ref{pretzel}.  
We provide a brief explanation below.

\begin{no interior singularity}
	For any hyperbolic L-space knot in $S^{3}$ with co-orientation-reversing monodromy,
	if the stable foliation of its monodromy has no interior singularities in the fibered surface,
	then every $3$-manifold obtained by Dehn surgery on this knot 
	admits a co-orientable taut foliation if and only if it is a non-L-space.
\end{no interior singularity}

\begin{proof}
	Suppose that $K_0$ is a hyperbolic L-space knot in $S^{3}$ with co-orientation-reversing monodromy,
	and we may assume that $K_0$ has right-veering monodromy.
	Let $g(K_0)$ denote the genus of $K_0$.
	As shown in the proof of Proposition \ref{pretzel} (b),  
	if the stable foliation of the monodromy of $K_0$ has no singularity in $\operatorname{Int}(S)$,
	then $K_0$ has degeneracy slope $4g(K_0)-2$.
	It follows from Corollary \ref{open book} that
	all rational slopes in $(-\infty,2g(K_0)-1)$ are CTF-surgery slopes for $K_0$.
	Since a slope $s \in \mathbb{Q} \cup \{\infty\}$ on $K$ is an NLS-surgery slope if and only if
    $s \in (-\infty,2g(K_0)-1)$ (see Subsection \ref{subsection 1.1}),
    it follows that
    every $3$-manifold obtained from Dehn surgery on $K_0$ admits a co-orientable taut foliation if and only if
	it is a non-L-space.
\end{proof}

\section{Acknowledgements}

The author wishes to thank Xingru Zhang for his guidance, 
patience,
encouragement,
and for many helpful discussions and comments on this work.
He is grateful to Nathan Dunfield for providing him the census of examples \hyperref[D3]{[D3]} and
answering several questions about the examples.
He thanks Cameron Gordon and Rachel Roberts for telling him 
the fractional Dehn twist coefficients of the $(-2,3,2q+1)$-pretzel knots and
some relevant works.
He thanks Chi Cheuk Tsang for telling him that
the monodromies of the $(-2,3,2q+1)$-pretzel knots have co-orientable stable foliations,
and for some helpful comments.
He thanks Cagatay Kutluhan, Johanna Mangahas, William Menasco,
Tao Li and Diego Santoro for some helpful conversations.
He thanks Steve Boyer for reviewing the manuscript and offering valuable suggestions.
He is grateful to the anonymous referee for valuable comments and suggestions that improved the paper.
He thanks Tech Topology Summer School 2023 and its organizers for providing him 
a great chance to communicate and discuss,
and for their travel support.

\appendix
\section{An alternative proof of Lemma \ref{transverse}}\label{appendix}

We present an alternative proof of Lemma \ref{transverse} from a different perspective as follows.
In this proof,
we analyze the leaf space of the singular $1$-dimensional foliation,
which is an $\mathbb{R}$-order tree, along with its maximal Hausdorff quotient.
For background on order trees of essential laminations,
we refer the reader to \hyperref[GO]{[GO, pp. 70-72, Construction of tree, Method II]}.
For order trees of planar essential laminations, see \hyperref[GK]{[GK]}.
For the maximal Hausdorff quotient, see \hyperref[RSS]{[RSS, Section 4]}.
We will clarify the differences between these references and our approach in Remark \ref{R-order tree}.

\begin{proof}[An alternative proof of Lemma \ref{transverse}]
	Let $\widetilde{\Sigma}$ be the universal cover of $\Sigma$,  
	and let $\widetilde{\mathcal{F}^{s}}$ be the pulled-back foliation of $\mathcal{F}^{s}$ in $\widetilde{\Sigma}$.  
	Let $\mathcal{H}$ be the restriction of $\widetilde{\mathcal{F}^{s}}$ to $\operatorname{Int}(\widetilde{\Sigma})$,  
	and let $L(\mathcal{H})$ be the leaf space of $\mathcal{H}$.  
	For convenience, for any $x \in \operatorname{Int}(\widetilde{\Sigma})$,  
	we denote by $\lambda_x$ the leaf of $\mathcal{H}$ containing $x$.  
	There is a canonical quotient map $q_{\mathcal{H}}: \operatorname{Int}(\widetilde{\Sigma}) \to L(\mathcal{H})$ such that  
	$q_{\mathcal{H}}(x) = q_{\mathcal{H}}(y)$ if and only if $\lambda_x = \lambda_y$.  
	We note that $L(\mathcal{H})$ is an $\mathbb{R}$-order tree and therefore is simply connected  
	(see Remark \ref{R-order tree} for an explanation).

	Recall that two leaves $\lambda, \mu$ of $\mathcal{H}$ are \emph{non-separated} if  
	there is no leaf of $\mathcal{H} - \{\lambda, \mu\}$ that separates them.  
	Equivalently, $\lambda, \mu$ are simultaneously approximated by a collection of leaves of $\mathcal{H}$.  
	It's not hard to observe that,
	for any two adjacent singularities on a component of $\partial \widetilde{\Sigma}$,
	if $\lambda, \mu$ are the two leaves of $\mathcal{H}$ such that 
	$\overline{\lambda}, \overline{\mu}$ contain these two singularities respectively,
	then $\lambda, \mu$ are non-separated.

	Let $L$ be the maximal Hausdorff quotient of $L(\mathcal{H})$,
	i.e. the quotient space of $L(\mathcal{H})$ under the closed equivalence relation generated by
	$\lambda \sim \mu$ if $\lambda, \mu$ are non-separated (see Remark \ref{R-order tree} for more information about $L$).
	We note that $L$ is still simply connected.
	Moreover, as observed in the previous paragraph,  
	any two leaves of $\mathcal{H}$ whose closures contain two adjacent singularities on $\partial \widetilde{\Sigma}$
	represent the same point in $L$.  
	It follows that any two leaves $\lambda, \mu$ of $\mathcal{H}$
	whose closures $\overline{\lambda}, \overline{\mu}$ intersect the same component of $\partial \widetilde{\Sigma}$
	also represent the same point in $L$.

	The map $q_{\mathcal{H}}$ induces a quotient map
	$q': \operatorname{Int}(\widetilde{\Sigma}) \to L$ such that
	$q'(x) = q'(y)$ if $\lambda_x, \lambda_y$ represent the same point in $L$.
	As explained above, 
	$q'(x) = q'(y)$ for all $x,y \in \operatorname{Int}(\widetilde{\Sigma})$ such that $\overline{\lambda_x}, \overline{\lambda_y}$
	intersect the same component of $\partial \widetilde{\Sigma}$.
	Thus, the map $q'$ can be extended to a quotient map
	$q: \widetilde{\Sigma} \to L$ such that $q \mid_{\operatorname{Int}(\widetilde{\Sigma})} = q'$,
	and for any component $C$ of $\partial \widetilde{\Sigma}$,
	$q$ takes all points in $C$ to $q'(\lambda)$ for some leaf $\lambda$ of $\mathcal{H}$
	with $\overline{\lambda} \cap C \neq \emptyset$.

	We note that the co-orientation on $\mathcal{F}^{s}$ induces 
	an orientation on $L(\mathcal{H})$,
	which also induces an orientation on $L$.
	
	Now, we may consider $\rho$ as a path $\rho: I \to \Sigma$ such that 
	(1)
	$\rho(0) = \rho(1)$ when $\rho$ is a simple closed curve, 
	and $\rho(0), \rho(1) \in \partial \Sigma$ when $\rho$ is a properly embedded arc,
	(2)
	$\rho$ can be divided into finitely many segments $\rho_1, \ldots, \rho_n$ which
	are either positively transverse to $\mathcal{F}^{s}$ or
	tangent to $\mathcal{F}^{s}$
	(where at least one of $\rho_1, \ldots, \rho_n$ is positively transverse to $\mathcal{F}^{s}$).
	Let $\widetilde{\rho}: I \to \widetilde{\Sigma}$ be a lift of $\rho$ to $\widetilde{\Sigma}$,
	and let $\widetilde{\rho_i}$ be the lift of $\rho_i$ to $\widetilde{\Sigma}$ contained in $\widetilde{\rho}$.
	Let $$\Omega = \{i \in \{1,\ldots,n\} \mid \rho_i \text{ is positively transverse to } \mathcal{F}^{s}\}.$$
	Then $\Omega \ne \emptyset$.
	In addition,
	each $\rho_i$ with $i \in \Omega$ is canonically identified with
	a positively oriented segment in $L$ via the quotient map $q: \widetilde{\Sigma} \to L$,
	and each $\rho_i$ with $i \in \{1,\ldots,n\} - \Omega$ is canonically identified with
	a single point of $L$.
	It follows that the path $\widetilde{\rho}$ in $\widetilde{\Sigma}$ can be canonically identified with
	a positively oriented path in $L$ (denoted $\gamma$) via the quotient map $q: \widetilde{\Sigma} \to L$.
	
	Since $L$ is simply connected, 
	the two endpoints of $\gamma$ are distinct,
	and therefore the map $q: \widetilde{\Sigma} \to L$ takes $\widetilde{\rho}(0), \widetilde{\rho}(1)$ to distinct points in $L$.
	As illustrated above,
	$\widetilde{\rho}(0), \widetilde{\rho}(1)$ cannot be contained in the same leaf of $\mathcal{F}^{s}$ or
	the same component of $\partial \widetilde{\Sigma}$.
	However, if $\rho$ is a non-essential simple closed curve,
	then the two endpoints $\widetilde{\rho}(0), \widetilde{\rho}(1)$ of $\widetilde{\rho}$ are identified.
	Similarly, if $\rho$ is a non-essential properly embedded arc,
	then the two endpoints $\widetilde{\rho}(0), \widetilde{\rho}(1)$ of $\widetilde{\rho}$ are contained in 
	the same component of $\partial \widetilde{\Sigma}$.
	Both of these cases contradict our previous conclusion.
	It follows that $\rho$ is essential.
\end{proof}

For completeness, we clarify the differences between the above proof and the referenced materials.  

\begin{remark}\label{R-order tree}\rm
	(a)
	We first explain why $L(\mathcal{H})$ is an order tree.
	Let $\mathscr{S}$ be the set of singular leaves of $\mathcal{H}$.
	We note that $\mathscr{S}$ is a countable union of leaves.
	By splitting open $\mathcal{H}$ along the union of leaves in $\mathscr{S}$,
	we obtain a lamination $\Lambda$ of $\operatorname{Int}(\widetilde{\Sigma})$.
	Note that $\operatorname{Int}(\widetilde{\Sigma})$ is homeomorphic to $\mathbb{R}^{2}$,
	and $\Lambda$ is an essential lamination on $\operatorname{Int}(\widetilde{\Sigma})$ in the sense of \hyperref[GK]{[GK]}.
	As described in \hyperref[GK]{[GK, Proposition 1.1]},
	$\Lambda$ is associated with an $\mathbb{R}$-order tree $T(\Lambda)$,
	such that there exists a quotient map $\nu: \operatorname{Int}(\widetilde{\Sigma}) \to T(\Lambda)$
	which takes each non-boundary leaf of $\Lambda$ 
	and the closure of each complementary region of $\Lambda$ to a point of $T(\Lambda)$.
	We can canonically identify $T(\Lambda)$ to the leaf space $L(\mathcal{H})$ of $\mathcal{H}$
	so that the quotient map $q_{\mathcal{H}}: \operatorname{Int}(\widetilde{\Sigma}) \to L(\mathcal{H})$ for $\mathcal{H}$
	is consistent with the map $\nu: \operatorname{Int}(\widetilde{\Sigma}) \to T(\Lambda)$.
	It follows that $L$ is an $\mathbb{R}$-order tree.
	
	(b)
	Here we provide some references on the maximal Hausdorff quotient of an $\mathbb{R}$-order tree.
	For a connected, simply connected, possibly non-Hausdorff $1$-manifold,
	the approach to obtain its maximal Hausdorff quotient is described in \hyperref[RSS]{[RSS, Section 4]},
	and the resulting space is referred to as its associated \emph{Hausdorff tree} in \hyperref[RSS]{[RSS, Definition 4.6]}.
	As described in \hyperref[RS]{[RS, page 186]} or \hyperref[Fe]{[Fe, page 4]},
	this operation also applies to $\mathbb{R}$-order trees.
\end{remark}

\end{document}